\documentclass[10pt]{amsart}

\usepackage{amsmath, amscd, amssymb}
\usepackage[frame,cmtip,arrow,matrix,line,graph,curve]{xy}
\usepackage{graphpap, color}
\usepackage[mathscr]{eucal}
\usepackage{cancel}
\usepackage{verbatim}

\numberwithin{equation}{section}

\newcommand{\CC}{\mathbb{C}}

\newcommand{\PP}{\mathbb{P}}
\newcommand{\QQ}{\mathbb{Q}}

\newcommand{\ZZ}{\mathbb{Z}}

\def\cO{{\cal O}}

\def\and{\quad{\rm and}\quad}

\newtheorem{prop}{Proposition}[section]
\newtheorem{theo}[prop]{Theorem}
\newtheorem{lemm}[prop]{Lemma}
\newtheorem{coro}[prop]{Corollary}
\newtheorem{rema}[prop]{Remark}
\newtheorem{exam}[prop]{Example}
\newtheorem{ques}[prop]{Question}
\newtheorem{defi}[prop]{Definition}

\def\beq{\begin{equation}}
\def\eeq{\end{equation}}

\def\PP{\mathbb{P}}
\def\CC{\mathbb{C}}

\def\lra{\longrightarrow}

\def\cO{\mathcal{O}}

\title[Fano visitors, Fano dimension and orbifold Fano hosts]{Fano visitors, Fano dimension and \\ orbifold Fano hosts}

\author{Young-Hoon Kiem and Kyoung-Seog Lee}

\address{Department of Mathematics and Research Institute
of Mathematics, Seoul National University, Seoul 151-747, Korea}
\email{kiem@math.snu.ac.kr}

\address{Center for Geometry and Physics, Institute for Basic Science (IBS), Pohang 37673, Republic of Korea}
\email{kyoungseog02@gmail.com}

\thanks{YHK was partially supported by NRF grant 2011-0027969; KSL was partially supported by IBS-R003-Y1.}

\begin{document}

\begin{abstract}
In \cite{KKLL}, the authors proved that every complete intersection smooth projective variety $Y$ is a Fano visitor, i.e. its derived category $D^b(Y)$ is equivalent to a full triangulated subcategory of the derived category $D^b(X)$ of a smooth Fano variety $X$, called a Fano host of $Y$. They also introduced the notion of Fano dimension of $Y$ as the smallest dimension of a Fano host $X$ and obtained an upper bound for the Fano dimension of each complete intersection variety. 

 In this paper, we provide a Hodge-theoretic criterion for the existence of a Fano host which enables us to determine the Fano dimensions precisely for many interesting examples, such as low genus curves, quintic Calabi-Yau 3-folds and general complete intersection Calabi-Yau varieties. 

Next we initiate a systematic search for more Fano visitors. We generalize the methods of \cite{KKLL} to prove that smooth curves of genus at most 4 are all Fano visitors and general curves of genus at most 9 are Fano visitors.
For surfaces and higher dimensional varieties, we find more examples of Fano visitors and raise natural questions. 

We also generalize Bondal's question and study triangulated subcategories of derived categories of Fano orbifolds. We proved that there are Fano orbifolds whose derived categories contain derived categories of orbifolds associated to quasi-smooth complete intersections in weighted projective spaces, Jacobians of curves, generic Enriques surfaces, some families of Kummer surfaces, bielliptic surfaces, surfaces with $\kappa=1,$ classical Godeaux surfaces, product-quotient surfaces, holomorphic symplectic varieties, etc.

An interesting recent discovery is the existence of quasi-phantom subcategories in derived categories of some surfaces of general type with $p_g=q=0$ (\cite{BBKS, BBS, GS, KKL, Lee1, Lee2, LS}).  But no examples of Fano with quasi-phantom have been found.
From the above constructions, we found Fano orbifolds whose derived categories contain quasi-phantom categories or phantom categories. 
\end{abstract}
\maketitle

\section{Introduction}

If one were to write up a list of keywords that describe recent developments in algebraic geometry, it would be hard to miss the words like ``derived category" or ``categorification" on the top part. 
The derived category $D^b(X)$ of bounded complexes of coherent sheaves of a projective variety $X$ was found to be a sophisticated invariant 
which categorifies geometric invariants such as  Hochschild homology, Hochschild cohomology and Grothendieck groups of algebraic varieties (cf. \cite{Kuz3}). Many geometric statements were categorified which means that a deep categorical origin or explanation was discovered.

One basic problem in algebraic geometry is to study 
how information of a given variety can be encoded in information of the other varieties. 
%
In 2011, Bondal 
raised the following question  (cf. \cite{BBF}).
\begin{ques}\label{q1.1} \emph{(Fano visitor problem)}\\
Let $Y$ be a smooth projective variety. Is there a Fano variety $X$ equipped with a fully faithful embedding $D^b(Y) \to D^b(X)$? 
\end{ques}
If the answer is yes, we call $Y$ a \emph{Fano visitor} and $X$ a \emph{Fano host} of $Y$. 

From the categorical point of view, Fano varieties are of particular interest because they admit natural semiorthogonal decompositions and many examples have been explicitly calculated  (cf. \cite{BBF, BO, Kuz1, Kuz2, ST}). They are one of the main objects in birational geometry and mirror symmetry. If the answer to Question \ref{q1.1} is yes for all smooth projective varieties, some problems about derived categories may be effectively reduced to those of Fano varieties. Moreover the geometry and invariants of $X$ are closely related to those of $Y.$ Especially, it turns out that moduli spaces of rational curves or vector bundles on $X$ are closely related to the geometry of $Y.$ See \cite{CKL, LMS1, LMS2, Reid} for such examples.

Bondal and Orlov in \cite{BO} proved that the derived category of a hyperelliptic curve $Y$ of genus $g$ is embedded into the derived category of the intersection of two quadrics in $\PP^{2g+1}$. Kuznetsov in \cite{Kuz1} proved that the derived categories of some K3 surfaces are embedded into special cubic 4-folds. He also discovered some Fano 3-folds that contain the derived categories of certain smooth projective curves  (cf \cite{Kuz2}).
Bernardara, Bolognesi and Faenzi in \cite{BBF} proved that every smooth plane curve is a Fano visitor. Segal and Thomas in \cite{ST} proved that a general quintic 3-fold is a Fano visitor by finding an 11-dimensional Fano host.

In \cite{KKLL}, the authors proved the following.

\begin{theo}\cite[Theorem 4.1]{KKLL}\label{t1.3}
All smooth projective complete intersections are Fano visitors.
\end{theo}
Moreover, they defined the \emph{Fano dimension} of a smooth projective variety $Y$ as the minimum dimension of Fano hosts $X$ of $Y$. The Fano dimension is defined to be infinite if no Fano hosts exist. It was also proved that an arbitrary complete intersection Calabi-Yau variety $Y$ of codimension $\le 2$ or a general complete intersection Calabi-Yau variety of codimension $\ge 3$ has Fano dimension at most $\dim Y+2$. 


 In this paper, we first provide a Hodge-theoretic criterion for the existence of a Fano host.
\begin{prop}\label{p1.4} (Proposition \ref{p2.18})
Let $Y$ be a Fano visitor and $X$ be a Fano host of $Y$. Then we have the inequality of Hodge numbers
$$\sum_{p-q=i} h^{p,q}(Y)\le \sum_{p-q=i} h^{p,q}(X)\quad \text{for all }  i.$$
\end{prop}
As a direct consequence, we obtain the following.
\begin{coro}\label{c1.5} (Corollary \ref{c4.2}) 
If $h^{n,0}(Y)\ne 0$ for $n=\dim Y > 0$, then the Fano dimension of $Y$ is at least $n+2$.
\end{coro}
Combining this corollary with the Fano host construction in \cite{KKLL}, we obtain the following.
\begin{coro}\label{c1.6}  (Corollary \ref{c4.5} and Proposition \ref{p4.6})\\
 The Fano dimension of a smooth projective curve of positive genus is at least 3. The Fano dimension of an arbitrary complete intersection Calabi-Yau variety $Y\subset \PP^{n+c}$ of codimension $c\le 2$ or a general complete intersection Calabi-Yau variety of codimension $c\ge 3$ is precisely $\dim Y+2$. 
\end{coro}
For instance, every smooth quintic 3-fold has Fano dimension 5 and the Fano host constructed in \cite{KKLL} has the minimal possible dimension.

Next we initiate a systematic search for more Fano visitors. We generalize the construction and technique of \cite{KKLL} for complete intersections in more general varieties such as Grassmannians (cf. Theorem \ref{t3.1}) or weighted projective spaces (cf. Theorem \ref{t.wci}). Using this, we prove that smooth curves of genus at most 4 are all Fano visitors (cf. \S\ref{s5.1}) and general curves of genus at most 9 are Fano visitors (cf. \S\ref{s5.2}).
For surfaces and higher dimensional varieties, we find more examples of Fano visitors and raise natural questions. 

We also generalize Bondal's question and study triangulated subcategories of derived categories of Fano orbifolds. Here a Fano orbifold means a smooth Deligne-Mumford stack(with trivial inertia in codimension at most 1) whose coarse moduli space is Fano variety. Conversely, it is known that for normal projective variety with quotient singularities we can find a smooth Deligne-Mumford stack whose coarse moduli space is the given variety. (cf. \cite{Kawamata1, Kawamata2, Vistoli, Popa}). We have following motivations for this generalization. 

First, it seems that it is hard to construct examples of smooth projective Fano varieties whose Grothendieck groups contain nontrivial finite abelian groups as direct summands. This might be an obstruction of Bondal's original Fano visitor problem.

By recent developments of the theory of Fano varieties, it seems that it is essential to consider Fano varieties having singularities. By works of Kawamata (cf. \cite{Kawamata1, Kawamata2, Kawamata3, Kawamata4}), it turns out that considering derived categories of smooth Deligne-Mumford stacks instead of considering derived categories of their coarse moduli spaces have many advantages. Therefore we investigate derived categories of Fano orbifolds instead of derived categories of Fano varieties having only quotient singularities. Moreover Fano orbifolds naturally appear in many context, e.g. mirror symmetry, orbifold Kahler-Einstein metric, etc. It will be nice if one can find a way to relate every Fano variety a smooth Artin stack whose coarse moduli space is the Fano variety, but it seems that we do not have such method yet.

Therefore a natural generalization of Bondal's original Fano visitor problem will be as follows.

\begin{ques}
(1) Which triangulated categories can be embedded into derived categories of smooth Deligne-Mumford stacks or smooth Artin stacks whose coarse moduli spaces are Fano? \\
(2) For a smooth projective variety $Y,$ is there a Fano orbifold $\mathcal{X}$ such that $D^b(\mathcal{X})$ contains $D^b(Y)$ as a full triangulated subcategory?
\end{ques}

In this paper, we restrict ourselves to consider only triangulated subcategories of Fano orbifolds and found many examples of varieties whose derived categories are contained in derived categories of Fano orbifolds.

\begin{defi}
Let $\mathcal{Y}$ be an algebraic stack. If there is a Fano orbifold $\mathcal{X}$ such that $D^b(\mathcal{X})$ contains $D^b(\mathcal{Y})$ then we say $\mathcal{Y}$ has an orbifold Fano host $\mathcal{X}.$
\end{defi}

Then we have the following result.

\begin{theo}
Every quasi-smooth weighted complete intersection in a weighted projective space is a coarse moduli space of a smooth Deligne-Mumford stack which has an orbifold Fano host. 
\end{theo}

Then it immediately follows that hyperelliptic curves and 95 families of (orbifold) K3 surfaces of Reid have orbifold Fano hosts. Moreover we can find many Fano orbifolds whose derived categories contain derived categories of many interesting varieties, e.g. Jacobians of curves, generic Enriques surfaces, some families of Kummer surfaces, bielliptic surfaces, surfaces with $\kappa=1,$ classical Godeaux surfaces, product-quotient surfaces, holomorphic symplectic varieties, etc. However we do not know whether there are smooth projective Fano varieties whose derived categories contain derived categories of these varieties.

An interesting recent discovery is the existence of quasi-phantom subcategories in derived categories of some surfaces of general type with $p_g=q=0$ (\cite{BBKS, BBS, GS, Lee1, Lee2, LS}).  But no examples of Fano with quasi-phantom have been found.
From the above constructions, we found Fano orbifolds whose derived categories contain quasi-phantom categories or phantom categories (cf. Example \ref{e6.13}). 

As far as we know, this is the first discovery of a  quasi-phantom category in the realm of Fano (orbifolds). However we do not know if there is a smooth Fano variety with a (quasi-)phantom category.


\medskip

Part of this work was done while the second named author was a research fellow of KIAS and visiting the University of Warwick by support of KIAS. He thanks KIAS and the University of Warwick for wonderful working conditions and kind hospitality. We thank Marcello Bernardara, Alexey Bondal, Chang-Yeon Chough, Alessio Corti, Enrico Fatighenti, Tomas Gomez, Atanas Iliev, In-Kyun Kim, Andreas Krug, Alexander Kuznetsov, Hwayoung Lee, Mudumbai Seshachalu Narasimhan, Dmitri Orlov, Shinnosuke Okawa, Genki Ouchi, Jihun Park, Miles Reid, Powel Sosna, Yukinobu Toda for helpful conversations.

\medskip

\noindent\textbf{Notation}. In this paper, all schemes and stacks are defined over the complex number field $\CC.$ For a vector bundle $E$ on $S,$ the projectivization $\PP E:=\mathrm{Proj}\left( \mathrm{Sym}^\cdot E^{\vee}\right)$ of $E$ parameterizes one dimensional subspaces in fibers of $E.$ For an algebraic stack $\mathcal{X},$ $D^b(\mathcal{X})$ denotes the bounded derived category of coherent sheaves on $\mathcal{X}.$ The zero locus $s^{-1}(0)$ of a section $s:\cO_X\to E$ of a vector bundle $E$ over a scheme $X$ is the closed subscheme of $X$ whose ideal is the image of $s^\vee:E^\vee\to \cO_X.$

\section{preliminaries}\label{s2}

In this section we recall several definitions and facts which we will use later.

\subsection{Fano varieties}\label{s2.1}

Let us recall several definitions on Fano varieties.

\begin{defi}\label{d2.1}\cite{IP}
A normal projective variety $X$ is Fano if $-K_X$ is $\mathbb{Q}$-Cartier and ample.
\end{defi}

It is well known that the Picard group of a Fano variety is a free abelian group (cf. \cite[Proposition 2.1.2]{IP}). 

\begin{defi}\label{d2.2}
Let $X$ be a smooth projective Fano variety. The largest positive integer $i$ which divides $K_X$ in $\mathrm{Pic}(X)$ is called the \emph{index} of $X$.
\end{defi}

Fano varieties have many nice properties.

\begin{theo}\label{t2.3}
For any positive integer $n$, there are only finitely many deformation equivalence classes of smooth projective Fano varieties of dimension $n$.
\end{theo}

\begin{theo}\cite[Theorem 2.2]{FH}\label{t2.4}
Fano varieties are rationally connected.
\end{theo}

\begin{theo}\cite[Corollary 4.18]{Deb}\label{t2.5}
Every smooth projective rationally connected variety is simply connected.
\end{theo}

Therefore we see that every smooth projective Fano variety is simply connected. Mori cones of weak Fano varieties are also very special.

\begin{theo}\cite[Theorem 2.3]{FH} \cite[Theorem 1.4]{Yasutake}\label{t2.6}
The Mori cone of a weak Fano variety is a rational polyhedral cone generated by classes of rational curves.
\end{theo}

If $X$ is a smooth Fano, $K_X^\vee$ is ample and hence $$H^{p,0}(X)\cong H^{0,p}(X)=H^p(X,\cO_X)=H^p(X,K_X\otimes K_X^{\vee})=0 \quad \text{for } p>0$$ by the Kodaira vanishing theorem. From Kamawama-Viehweg vanishing theorem we have similar results for singular Fano varieties having at most log terminal singularities. Therefore we obtain the following vanishing of Hodge numbers.
\begin{lemm} \cite{IP} \label{l2.7}
If $X$ is a smooth projective Fano variety, $h^{p,0}(X)=0$ for $p>0$. If $X$ is a singular Fano variety having at most log terminal singularities and $\widetilde{X}$ is its resolution of singularities. Then we have $h^{p,0}(\widetilde{X})=h^{p,0}(X)=0$ for $p>0$.
\end{lemm}



\subsection{Derived categories of coherent sheaves on algebraic stacks}

In this section we collect some necessary facts about derived categories of coherent sheaves on algebraic stacks. Because of the lack of spaces we do not recall all basic definitions and properties about algebraic stacks and refer \cite{Gomez, Kawamata1, Kawamata2, Olsson, Vistoli} for these backgrounds. Unless otherwise stated, we will work on the category of schemes over $\mathbb{C}$ with big $\acute{e}$tale topology.

Let $\mathcal{C}$ be the category of schemes over $\mathbb{C}.$ Then Yoneda Lemma enables us to understand the category of schemes as a full subcategories of sheaves on $\mathcal{C}.$ Then we can enlarge the category of schemes by algebraic spaces. Let us recall the definition of algebraic spaces.

\begin{defi}\cite{Olsson}
Let $\mathcal{C}$ be the category of schemes over $\mathbb{C}.$ An algebraic space over $\CC$ is a functor $X : \mathcal{C}^{op} \to Set$ such that \\
(1) $X$ is a sheaf with respect to the big $\acute{e}$tale topology. \\
(2) $\Delta : X \to X \times X$ is representable by schemes. \\
(3) There exist an $\acute{e}$tale surjective morphism $\pi : U \to X$ with $U$ a scheme over $\CC.$ \\
Morphisms between algebraic spaces are natural transformations.
\end{defi}

In many moduli problems the notion of algebraic space is not enough to represent the moduli functors. Often this is because the families we consider have nontrivial automorphisms. Therefore we need to enlarge the notion of sheaves to the notion of categories fibered in groupoids. Roughly speaking, stacks are categories fibered in groupoids whose descent data are effective. Algebraic stacks are stacks where we can naturally extend many definitions and properties of schemes. To be more precise, we can define algebraic stacks as follows.

\begin{defi}\cite{Olsson} Let $\mathcal{C}$ be the category of schemes over $\mathbb{C}.$ \\
(1) A category fibered in groupoids $p : \mathcal{X} \to \mathcal{C}$ is a stack if for every object $V \in \mathcal{C}$ and coverings $\{ V_i \to V \}_{i \in I}$, the functor
$$ \mathcal{X}(V) \to \mathcal{X}(\{ V_i \to V \}_{i \in I}) $$
is an equivalence of categories where $\mathcal{X}(\{ V_i \to V \}_{i \in I})$ are descent data associated to the covering $\{ V_i \to V \}_{i \in I}.$ \\
(2) A stack $\mathcal{X}$ is an Artin stack if the diagonal $ \Delta : \mathcal{X} \to \mathcal{X} \times \mathcal{X}$
is representable and there exist a smooth surjective morphism $\pi : U \to \mathcal{X}$ with $U$ a scheme over $\CC.$ \\
(3) An Artin stack $\mathcal{X}$ is a Deligne-Mumford stack if there is an $\acute{e}$tale surjection $U \to \mathcal{X}$ with $U$ a scheme. \\
We call such $U$ an atlas of $\mathcal{X}.$
\end{defi}

In many moduli problems, there are schemes which corepresent the moduli functors. We call these schemes(more generally algebraic spaces) the coarse moduli spaces of the moduli functors. Let us recall the definition of coarse moduli spaces of algebraic stacks as follows.

\begin{defi}\cite{Olsson}
Let $\mathcal{X}$ be an algebraic stack. A coarse moduli space of $\mathcal{X}$ is a morphism $\pi : \mathcal{X} \to X$ from $\mathcal{X}$ to an algebraic space $X$ such that \\
(1) If $\pi' : \mathcal{X} \to X'$ is a morphism to an algebraic space $X'$ then  there exists a unique morphism $f : X \to X'$ such that $\pi'=\pi \circ f,$\\
(2) For every algebraically closed field $k$ there is an induced map $|\mathcal{X}(k)| \to X(k)$ which is a bijective, where $|\mathcal{X}(k)|$ denotes the set of isomorphism classes in $|\mathcal{X}(k)|.$
\end{defi}

For every nomal projective variety $X$ having at worst quotient singularities, there is a smooth Deligne-Mumford stack $\mathcal{X}$ whose coarse moduli space is $X.$ See \cite{Kawamata1, Kawamata2, Vistoli} for details. In many cases, the coarse moduli space determines the smooth Deligne-Mumford stack. Let us recall the following fact.

\begin{prop}\cite[Proposition 2.8]{Vistoli}
A smooth seperated Deligne-Mumford stack whose inertia is trivial in codimension at most 1 is determined by its coarse moduli space.
\end{prop}

A coherent sheaf $\mathcal{F}$ on an algebraic stack $\mathcal{X}$ is data which assign a coherent sheaf $\mathcal{F}_U$ for each scheme $U$ over $\mathcal{X}$ and this assignment should satisfy cocycle conditions. To be more precise, let us recall the definition as follows. Here we restrict ourselves to consider only coherent sheaves on Deligne-Mumford stacks. For more general definition, see \cite{Olsson}.

\begin{defi} \cite{Vistoli}
Let $\mathcal{X}$ be a Deligne-Mumford stack. A quasi-coherent sheaf $\mathcal{F}$ on $\mathcal{X}$ is the following data. \\
(1) For any atlas $U \to \mathcal{X},$ a quasi-coherent sheaf $\mathcal{F}_U.$ \\
(2) An isomorphism $\phi_f : f^*\mathcal{F}_V \to \mathcal{F}_U$ for any following commutative diagram
\[\xymatrix{
U \ar[rd] \ar[rr]^f & & V \ar[ld] \\
 & \mathcal{X} &
}\]
where $U, V$ are atlases of $\mathcal{X}$ and $f: U \to V$ be a morphism of schemes.
The above data satisfy the following cocyle condition. For three atlases $U,V,W$ and the following commutative diagram
\[\xymatrix{
U \ar[rd] \ar[r]^f & V \ar[d] \ar[r]^g & W \ar[ld] \\
& \mathcal{X} &
}\]
we have the following commutative diagram.
\[\xymatrix{
\mathcal{F}_U \ar[rd]^{\phi_f} \ar[rr]^{\phi_{g \circ f}} & & (g \circ f)^* \mathcal{F}_W = f^*g^* \mathcal{F}_W \\
 & f^*\mathcal{F}_V \ar[ru]^{f^*\phi_g} &
}\]
We say $\mathcal{F}$ a coherent(resp. locally free) sheaf if $\mathcal{F}_U$ is coherent(resp. locally free) sheaf for every $U.$ 
\end{defi}

Let us recall the famous theorem of Keel and Mori. 

\begin{theo}\cite{Olsson}\label{KeelMori}
Let $\mathcal{X}$ be an algebraic stack which is locally of finite presentation over $\mathbb{C}$ with finite diagonal. Then there exists a coarse moduli space $\mathcal{X} \to X$ such that \\
(1) $X$ is locally of finite type, and if $\mathcal{X}$ is separated, then $X$ is also separated, \\
(2) $\pi$ is proper and $\mathcal{O}_X \to \pi_*\mathcal{O}_{\mathcal{X}}$ is an isomorphism, \\
(3) If $U \to X$ is a flat morphism of algebraic spaces, then the natural map $\mathcal{X} \times_X U \to U$ is a coarse moduli space of $\mathcal{X} \times_X U.$
\end{theo}

Sometimes we need to compare derived categories of algebraic stacks and their coarse moduli spaces. Let us recall the following result.

\begin{prop}\cite[Proposotion 11.3.4]{Olsson}\label{exact}
Let $\mathcal{X}$ be a locally finite type Deligne-Mumford stack with finite diagonal. Then $\pi_*$ is an exact functor.
\end{prop}

Then we can have the following.

\begin{lemm}\label{l.smoothcoarsemoduli}
Let $\mathcal{X}$ be a locally noetherian Deligne-Mumford stack with finite diagonal and $X$ be its coarse moduli space. Suppose that $X$ is a smooth projective variety. Then we have a fully faithful functor $L\pi^* : D^b(X) \to D^b(\mathcal{X}).$
\end{lemm}
\begin{proof}
From the above Theorem \ref{KeelMori} and Proposition \ref{exact}, we have an isomorphism $\mathcal{O}_X \to R\pi_* \mathcal{O}_{\mathcal{X}} .$
Because $X$ is a smooth projective variety we know that every object in $D^b(X)$ is a perfect complex. Then we have a canonical isomorphisms $Hom^k(L\pi^* a,L\pi^* b) \cong Hom^k(a,R\pi_* L\pi^* b) \cong Hom^k(a,b)$ for any $k$ and $a,b$ which are objects in $D^b(X)$ by the adjuction formula and the projection formula (cf. \cite{HR, Olsson}). From \cite{Huy2} we see that $L\pi^*$ is a fully faithful functor.
\end{proof}

Let us consider a type of examples of algebraic stacks which will play a key role in our construction of orbifold Fano hosts.

\begin{exam}\cite{Vistoli}
Let $X$ be a scheme and $G$ be a reductive algebraic group acting on $X.$ Then the quotient stack $[X/G]$ is defined as follows. \\
(1) An object of $[X/G]$ is  a principal $G$-bundle $E \to S$ with a $G$-equivariant morphism $E \to X.$ \\
(2) A morphism between $E_1 \to S_1$ with $E_2 \to S_2$ is the following commutative diagrams
\[\xymatrix{
E_1 \ar[d] \ar[r]^f & E_2 \ar[d] \\
S_1 \ar[r] & S_2
}\]
and 
\[\xymatrix{
E_1 \ar[rd] \ar[rr]^f & & E_2 \ar[ld] \\
 & X &
}\]
where $f$ is a $G$-equivariant morphism.
\end{exam}

The coherent sheaves and coarse moduli spaces of the quotient stacks are well-known as follows.

\begin{rema}
Let $X$ be a scheme and $G$ be an algebraic group acting on $X.$ Then the category of coherent sheaves on $[X/G]$ is equivalent to the category of $G$-equivariant sheaves on $X.$ See \cite{Olsson, Vistoli} for more details.
\end{rema}


\begin{prop}\cite[Proposition 2.11]{Vistoli}\label{coarsemoduliquotientstack}
Let $X$ be a scheme and $G$ be a reductive algebraic group acting on $X.$ Suppose that there is a geometric quotient $X/G$ and the quotient map is universally submersive. Then there is a natural morphism $\pi : [X/G] \to X/G$ which is the coarse moduli space of $[X/G].$
\end{prop}

Quotient stacks of Fano varieties by finite groups form a natural class of Fano orbifolds.

\begin{coro}
Let $X$ be a smooth Fano variety and $G$ be a finite group acting on $X.$ Suppose that the locus with nontrivial stabilizer on $X$ has codimension at least 2. Then $[X/G]$ is a Fano orbifold.
\end{coro}
\begin{proof}
It is easy to see that $[X/G]$ is a smooth Deligne-Mumford stack. From the above Proposition \ref{coarsemoduliquotientstack}, we see that $X/G$ is a coarse moduli space of $[X/G].$ From the assumption we see that the canonical bundle of $X$ is the pullback of the canonical bundle of $X/G.$ Therefore the anticanonical bundle of $X/G$ is ample.
\end{proof}

Now let us recall the notion of inertia stack and orbifold cohomology. We will follow explanation of \cite{Popa}. 

\begin{defi}\cite{Olsson}
Let $\mathcal{X}$ be an algebraic stack. The inertia stack $\mathcal{I}_{\mathcal{X}}$ of $\mathcal{X}$ is the fiber product of the following diagram.
\[\xymatrix{
\mathcal{I}_{\mathcal{X}} \ar[d] \ar[r] & \mathcal{X} \ar[d]^{\Delta_{\mathcal{X}}} \\
\mathcal{X} \ar[r]^{\Delta_{\mathcal{X}}} & \mathcal{X} \times \mathcal{X}
}\]
\end{defi}

Let $\mathcal{X}$ be an $n$-dimensional Deligne-Mumford stack, $\mathcal{Z}$ be a component of $\mathcal{I}_{\mathcal X},$ $(z,g)$ be a generic point of $\mathcal{Z}$ where $g \in Aut(z)$ and $m$ be an order of $g.$ Because $\langle g \rangle$ is an abelian group acting on $T_z \mathcal{X},$ we can write $g$ acts on $T_z \mathcal{X}$ as $ diag(\epsilon^{a_1},\cdots,\epsilon^{a_n}) $ where $\epsilon$ be a primitive $m$-th root of unity and $1 \leq a_i \leq m.$ We can define the shift number $a(\mathcal{Z})$ of $\mathcal{Z}$ as follows.
$$ a(\mathcal{Z}) = n - \frac{1}{n} \sum_{i=1}^{n}{a_i} $$

Then we can define orbifold cohomology of Deligne-Mumford stack as follows.

\begin{defi}
Let $\mathcal{X}$ be a Deligne-Mumford stack. The orbifold Hodge number is 
$$ h_{orb}^{p,q}(\mathcal X) = \sum_{\mathcal{Z} \subset \mathcal{I}_{\mathcal X}} h^{p-a(\mathcal Z),q-a(\mathcal Z)}(Z) $$
where $\mathcal{Z}$ is a component of $\mathcal{I}_{\mathcal X},$ $Z$ is the coarse moduli space of $\mathcal{Z}$ and $a(\mathcal Z)$ is the shift number of $\mathcal{Z}.$
\end{defi}

See \cite{Popa, Yasuda1, Yasuda2} for more details about orbifold cohomology.

Let us recall derived McKay correspondence which will be very useful to construct orbifold Fano hosts of many interesting varieties.

\begin{theo}\cite{BKR}\label{t.BKR}
Let $X$ be an $n$-dimensional smooth projective variety, $G$ be a finite group acting on $X,$ $Y$ be the the component of $G$-Hilbert scheme containing free $G$-orbits and $Z$ be the universal family. Suppose that for every $x \in X,$ the stabilizer group is a finite subgroup of $SL(n,T_xX)$ and dimension of $Y \times_{X/G} Y$ is less than or equal to $n+1.$ 
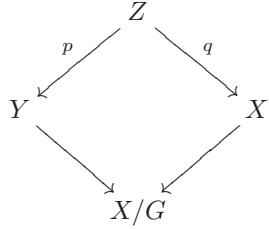
\[\xymatrix{
 & Z \ar[ld]_p \ar[rd]^q & \\
Y \ar[rd]  &  &  X \ar[ld] \\
 &  X/G  &
}\]
Then $Y$ is a crepant resolution of $X/G$ and we have an equivalence $Rq_*Lp^* : D^b(Y) \simeq D^b([X/G]).$
\end{theo}

From the above theorem and a theorem of Haiman (cf. \cite{Haiman}) we can also obtain the following result.

\begin{theo}\cite{BKR, Haiman}\label{t.BKRH}
Let $Y$ be a smooth projective algebraic surface. Then $Y^n$ has a natural $S_n$-action. Let $Y^{[n]}$ be the Hilbert scheme of $n$-points of $Y.$ Then $D^b(Y^{[n]}) \simeq D^b([Y^n/S_n]).$
\end{theo}

We have similar results for $G \subset GL(2,\CC)$ as follows.

\begin{theo} \cite{Ishii, IU}\label{t.GL2}
Let $X$ be a smooth projective surface, $G$ be a finite group acting on $X$ and $Y$ be the minimal resolution of $X/G.$ Suppose that for every $x \in X,$ the stabilizer group is a finite subgroup of $GL(2,T_xX).$
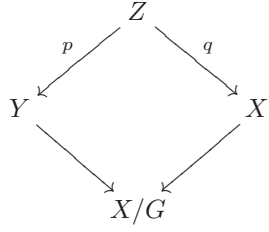
\[\xymatrix{
 &  Z \ar[ld]_p \ar[rd]^q & \\
Y \ar[rd]  &  &  X \ar[ld] \\
 &  X/G  &
}\]
Then we have a fully faithful embedding $Rq_*Lp^* : D^b(Y) \to D^b([X/G]).$
\end{theo}

\subsection{Weighted projective spaces}

Let $\bar{a}=(a_0,a_1,\cdots,a_n)$ be a sequence of positive integers. Then consider the graded polynomial ring $\CC[z_0,\cdots,z_n]$ with degree of $z_i=a_i.$ Then we can define the weighted projective space $\PP(\bar{a})$ as the projective variety $Proj(\CC[z_0,\cdots,z_n]).$ Note that giving a grading on $\CC[z_0,\cdots,z_n]$ corresponds to giving a $\CC^*$-action on $\CC^{n+1}.$ We also define $\mathcal{P}(\bar{a})$ to be the smooth Deligne-Mumford stack $[\mathbb{C}^{n+1} - \{0 \} / \mathbb{C}^*]$ whose coarse moduli space is $\PP(\bar{a}).$ They provide examples of Fano orbifolds.

When we consider derived categories then it is nicer to consider a weighted projective space as a smooth Deligne-Mumford stack $\mathcal{P}(\bar{a}).$ However when we do geometry it is easier to consider a weighted projective space as a projective variety $\PP(\bar{a})$ although there are several pathological phenomena. See \cite{Dolgachev} for more details.

Let us recall some of relevant definitions.

\begin{defi}\cite{Dolgachev}
(1) For a closed subscheme $Y \subset \PP(\bar{a}),$ we can associate a quasi-cone $C_Y$ which is the scheme closure of the inverse of $Y$ in the $\CC^{n+1}.$ Let us also denote $C_Y^*$ to be $C_Y - \{ 0 \}$. \\
(2) A closed subscheme $Y \subset \PP(\bar{a})$ is called quasi-smooth if $C_Y$ is smooth outside of its vertex. \\
(3) $Y$ is a weighted complete intersection of multidegree $\bar{d}=(d_1,\cdots,d_c)$ if $I_Y$ is generated by a regular sequence of homogeneous elements $f_1,\cdots,f_c$ where the degree of $f_i$ is $d_i.$
\end{defi}

For every quasi-smooth weighted complete intersection $Y$ we consider its associated stack $\mathcal{Y}=[C_Y^*/\mathbb{C}^*].$ It is a smooth Deligne-Mumford stack whose coarse moduli space $Y.$ 
Recall that $\PP(\bar{a})$ is a quotient of $\PP^n$ by $\mu_{\bar{a}}$-action where $\mu_{\bar{a}}$ acts on $\PP^n$ via 
$$ (\mu_0,\cdots,\mu_n) \cdot [z_0:\cdots:z_n]=[\epsilon _0^{\mu_0} \cdot z_0 : \cdots : \epsilon _n^{\mu_n} \cdot z_n] $$
and $\epsilon_i$ is a primitive $a_i$-th root of unity.  
By this covering we can also define a smooth Deligne-Mumford stack. See \cite{Kawamata2} for more details. 

The sheaves $\mathcal{O}_{\PP(\bar{a})}(n)$ have some common properties with $\mathcal{O}_{\PP(1,\cdots,1)}(n)$ as follows.

\begin{rema}\cite{BR, Dolgachev}
(1) The sheaf $\mathcal{O}_{\PP(\bar{a})}(n)$ is reflexive. \\
(2) Let $a=l.c.m. \{ a_0,\cdots,a_n \}.$ Then $\mathcal{O}_{\PP(\bar{a})}(a)$ is invertible.
\end{rema}

However the properties of sheaves $\mathcal{O}_{\PP(\bar{a})}(n)$ can be very different from $\mathcal{O}_{\PP(1,\cdots,1)}(n).$ Let us mention  several such properties.

\begin{rema}\cite{BR, Dolgachev}
(1) $\mathcal{O}_{\PP(\bar{a})}(n)$ may not be invertible. \\
(2) $\mathcal{O}_{\PP(\bar{a})}(n)$ may be invertible but not ample even if $n>0.$ \\
(3) $H^0(\PP(\bar{a}),\mathcal{O}_{\PP(\bar{a})}(1))=0$ if and only if $a_i > 1$ for all $i.$ \\
(4) $\mathcal{O}_{\PP(\bar{a})}(n_1)$ and $\mathcal{O}_{\PP(\bar{a})}(n_2)$ can be isomorphic even if $n_1 \neq n_2.$ \\
(5) $\mathcal{O}_{\PP(\bar{a})}(n_1) \otimes \mathcal{O}_{\PP(\bar{a})}(n_2) \to \mathcal{O}_{\PP(\bar{a})}(n_1+n_2)$ may not be an isomorphism. 
\end{rema}

Therefore we need to be careful to extend results about ordinary projective space to results about weighted projective spaces.

\subsection{Semiorthogonal decomposition}\label{s2.2}

We recall the definition and examples of semiorthogonal decompositions of derived categories of coherent sheaves.

\begin{defi}\label{d2.9}
Let $\mathcal{T}$ be a triangulated category. A \emph{semiorthogonal decomposition} of $\mathcal{T}$ is a sequence of full triangulated subcategories $ \mathcal{A}_1, \cdots, \mathcal{A}_n $ satisfying the following properties:
\\
(1) $Hom_{\mathcal{T}}(a_i, a_j)=0$ for any $a_i \in \mathcal{A}_i, a_j \in \mathcal{A}_j$ with $i > j$; \\
(2) the smallest triangulated subcategory of $\mathcal{T}$ containing $ \mathcal{A}_1, \cdots, \mathcal{A}_n $ is $\mathcal{T}.$ \\
We will write $\mathcal{T} = \langle \mathcal{A}_1, \cdots, \mathcal{A}_n \rangle$ to denote the semiorthogonal decomposition.
\end{defi}

Let $E$ be a vector bundle of rank $r\ge 2$ over a smooth projective variety $S$ and let $Y=s^{-1}(0)\subset S$ denote the zero locus of a regular section $s \in H^0(S,E)$ such that $ \dim Y = \dim S - \mathrm{rank}\, E$. 
Let $X=w^{-1}(0) \subset \PP E^\vee$ be the zero locus of the section $w\in H^0(\PP E^\vee, \cO_{\PP E^\vee}(1))$  
determined by $s$ under the natural isomorphisms
$$H^0(\PP E^\vee, \cO_{\PP E^\vee}(1))\cong H^0(S, q_*\cO_{\PP E^\vee}(1))\cong H^0(S,E)$$
where $q:\PP E^\vee\to S$ is the projection map of the projective bundle.

Orlov proved in \cite{Orlov2} that $D^b(X)$ has the following semiorthogonal decomposition which was subsequently generalized to higher degree hypersurface fibrations by Ballard, Deliu, Favero, Isik and Katzarkov in \cite{BDFIK}.

\begin{theo}\cite[Proposition 2.10]{Orlov2}\label{t2.10} There is a natural semiorthogonal decomposition
$$ D^b(X)= \langle q^*D^b(S), \cdots, q^*D^b(S) \otimes_{\cO_X} {\cO_X}(r-2), D^b(Y) \rangle .$$
\end{theo}

\begin{rema}\label{r2.11}
Orlov proved in particular that there is a fully faithful exact functor from $D^b(Y)$ to $D^b(X)$ (cf. \cite[Proposition 2.2]{Orlov2}). When an algebraic group $G$ acts on $S$ and $E$ compatibly and $s$ is a $G$-invariant section, there is an induced action of $G$ on $X$ and $Y$. His proof also works for this equivariant setting to give us a fully faithful exact functor from $D^b([Y/G])$ to $D^b([X/G])$. See \cite[Remark 2.9]{Orlov2}.
\end{rema}

We can provide many examples of orbifold Fano hosts of interesting algebraic varieties using the following result of Ploog in \cite{Ploog} which was generalize by Krug and Sosna in \cite{KS}.

\begin{theo}\cite{KS, Ploog}\label{t.KSP}
Let $X, Y$ be smooth projective varieties with $G$-action where $G$ is a finite group. Suppose that $\Phi_K : D^b(Y) \to D^b(X)$ is a fully faithful functor and $K$ has a $G$-linearization with respect to the diagonal $G$-action on $Y \times X.$ Then $K$ induces a functor $\Phi^G_K : D^b([Y/G]) \to D^b([X/G])$ which is also fully faithful. 
\end{theo}

\subsection{Fano visitor problem}

We learned the definition of Fano visitor from \cite{BBF}. 

\begin{defi}
An algebraic stack $\mathcal{Y}$ is called a \emph{Fano visitor} if there is a smooth projective Fano variety $X$ together with a fully faithful (exact) embedding $D^b(\mathcal{Y}) \to D^b(X)$. We call such a Fano $X$ a \emph{Fano host} of $\mathcal{Y}.$ If there is a smooth Deligne-Mumford stack $\mathcal{X}$ whose coarse moduli space is Fano and $D^b(\mathcal{X})$ contains $D^b(\mathcal{Y})$ as a full triangulated subcategory, then $\mathcal{X}$ is called an orbifold Fano host of $\mathcal{Y}.$
\end{defi}

\begin{rema}
Let $Y$ be a singular variety. Then there are objects $e_1$ and $e_2$ such that $Hom(e_1,e_2[i])$ is nonzero for infinitely many $i.$ Therefore there is no Fano orbifold $\mathcal{X}$ such that $D^b(\mathcal{X})$ contains $D^b(Y)$ as a full triangulated subcategory.  
\end{rema}

Bondal's question (Question \ref{q1.1}) asks if a smooth projective variety is a Fano visitor.
It is easy to see that a Fano host $X$ of a smooth projective variety $Y$ is not unique because for instance the product $X\times \PP^1$ is also a Fano host of $Y$. So we may ask for a Fano host of minimal dimension.  

\begin{defi}\cite{KKLL}
The \emph{Fano dimension} of a smooth projective variety $Y$ is the minimum among the dimensions $\dim X$ of Fano hosts $X$ of $Y$.
\end{defi}
See \cite{KKLL} for more discussions and questions related to Fano visitors.

\section{Cayley's trick and weighted complete intersections}\label{s3}

In this section, we recall and generalize the main construction and result in \cite{KKLL}.

\medskip

\subsection{Cayley's trick}\label{s3.1}
Let $S$ be a smooth variety and $s\in H^0(S,E)$ be a regular section of a vector bundle of rank $r\ge 2$ such that $Y=s^{-1}(0)$ is smooth of dimension $\dim S-r$. Let $\PP E^\vee=\mathrm{Proj} \left( \mathrm{Sym}^\cdot E\right)$ 
denote the projectivization of $E^\vee$. Then we have an isomorphism
$$H^0(S,E)\cong H^0(\PP E^\vee,\cO_{\PP E^\vee}(1))$$
which gives us a section $w$ of $\cO_{\PP E^\vee}(1)$ corresponding to $s$. Let $X=w^{-1}(0)$. Since $Y$ is smooth, $X$ is also smooth by local computation. We have the following commutative diagram

\[\xymatrix{
\PP N^\vee \ar[d]_p \ar[r]^i & X \ar[r] & \PP E^\vee \ar[d] \\
Y \ar[rr] & & S
}\]

where $N$ is the normal bundle.
By Orlov's theorem (cf. Theorem \ref{t2.10}), there is a fully faithful embedding $Ri_*Lp^* :D^b(Y) \to D^b(X)$. Therefore if $X$ is Fano, then $X$ is a Fano host of $Y$ and $Y$ is a Fano visitor.

Note that there is an embedding $\PP N^\vee \to Y \times X$ induced from the above diagram and the functor $Ri_*Lp^* : D^b(Y) \to D^b(X)$ is a Fourier-Mukai transform $\Phi_K$ whose kernel $K$ is $\mathcal{O}_{\PP N^\vee}.$ Suppose that there is an algebraic group $G$ acting on $Y$ and the action extends to $S$ and $E$ and $Y$ is given by an invariant section $s.$ Then $G$ acts on $\PP N^\vee$ and $\mathcal{O}_{\PP N^\vee}$ has a canonical $G$-linearization induced by the group action. When $G$ is a finite group, we can recover the Remark \ref{r2.11} of Orlov from the Theorem \ref{t.KSP}. Moreover it holds when $G$ is a reductive algebraic group (cf. \cite[Remark 2.9]{Orlov2}).

\medskip

\subsection{Complete intersections in projective space}\label{s3.2}

When $Y\subset \PP^m$ is a smooth complete intersection defined by a section $s'$ of $\oplus_{i=1}^l\cO_{\PP^m}(a_i)$ with $a_i>0$ and $l\ge 0$, we enlarge the ambient space $\PP^m$ to $\PP^{m+c}=S$ and extend the vector bundle 
$\bigoplus_{i=1}^l\cO_{\PP^{m}}(a_i)$ to $$\bigoplus_{i=1}^l\cO_{\PP^{m+c}}(a_i)\oplus \cO_{\PP^{m+c}}(1)^{\oplus c}=E$$
for $c\ge 0$.
The section $s'$ together with a choice of defining linear equations for $\PP^m\subset \PP^{m+c}$ gives us a section $s$ of $E$ with $s^{-1}(0)={s'}^{-1}(0)=Y$. Applying Cayley's trick above, we obtain a hypersurface $X=w^{-1}(0)$ of $\PP E^\vee$ whose dimension is $m+2c+l-2=\dim Y+2c+2l-2$. 

The authors proved in \cite[\S4.2]{KKLL} that if $c$ is greater than $\sum_{i=1}^la_i-m-l$ and $1-l$, then $X$ is Fano. This proves the main result (Theorem \ref{t1.3}) of \cite{KKLL} because $X$ is a Fano host of $Y$ by the discussion in \S\ref{s3.1}.

\medskip

\subsection{A generalization}\label{s3.3}
We can capture the essence of the proof of  Theorem \ref{t1.3} in \cite{KKLL} as follows.

\begin{theo}\label{t3.1}
We use Cayley's trick in \S\ref{s3.1}. Suppose that
\begin{enumerate}
\item $E$ is ample and  $K_S^\vee \otimes \det E^\vee$ is nef, or
\item there is a nef line bundle $H$ such that $F:=E \otimes H^{\vee}$ is a nef vector bundle and that $K_S^\vee \otimes \det E^\vee \otimes H^{r-1}$ is ample. 
\end{enumerate}
Then $X=w^{-1}(0)$ is a Fano host of $Y=s^{-1}(0)$.
\end{theo}
\begin{proof}
By Theorem \ref{t2.10}, it suffices to show that $X$ is Fano. For (1), see \cite[Lemma 3.1]{KKLL}. For (2), let $q : \PP E^\vee \to S$ denote the canonical projection.
Let us compute $K_X$. From the relative Euler sequence 
$$0\lra \cO_{\PP E^\vee}\lra q^*E^\vee\otimes \cO_{\PP E^\vee}(1) \lra T_{\PP E^\vee/S}\lra 0,$$
we have $K_{\PP E^\vee/S}^\vee=(q^*\det E^\vee)\otimes \cO_{\PP E^\vee}(r)$. From $K_{\PP E^\vee}=q^*K_S\otimes K_{\PP E^\vee/S}$ we have
$$ K_{\PP E^\vee}^\vee \cong q^*(K_S^\vee \otimes \det E^\vee) \otimes \cO_{\PP E^\vee}(r).$$
Therefore we get 
$$K_X^\vee=K_{\PP E^\vee}^\vee \otimes \cO(-1)|_X\cong q^*(K_S^\vee \otimes \det E^\vee) \otimes \cO_{\PP E^\vee}(r-1)|_X$$
$$ \cong q^*(K_S^\vee \otimes \det E^\vee \otimes H^{r-1} ) \otimes \cO_{\PP F^\vee}(r-1)|_X. $$
By assumption, both $q^*(K_S^\vee \otimes \det E^\vee \otimes H^{r-1} )$ and $ \cO_{\PP F^\vee}(r-1) $ are nef line bundles, and so is $K_X^\vee$.
To see that $K_X^\vee$ is big, let us compute the intersection number $(K_X^\vee)^{\dim X}$ as follows:
$$ (K_X^\vee)^{\dim X}=(q^*(K_S^\vee \otimes \det E^\vee \otimes H^{r-1} ) \otimes \cO_{\PP F^\vee}(r-1)|_X)^{\dim X} $$ 
$$ =(q^*(K_S^\vee \otimes \det E^\vee \otimes H^{r-1} ) \otimes \cO_{\PP F^\vee}(r-1))^{\dim X} \cdot \cO_{\PP E^\vee}(1) $$
$$ =(q^*(K_S^\vee \otimes \det E^\vee \otimes H^{r-1} ) \otimes \cO_{\PP F^\vee}(r-1))^{\dim X} \cdot (q^*H \otimes \cO_{\PP F^\vee}(1)). $$
By the binomial expansion formula, we see that $(K_X^\vee)^{\dim X}$ is positive since every term is a multiple of a nef line bundle and $q^*(K_S^\vee \otimes \det E^\vee \otimes H^{r-1} )^{\dim S} \cdot \cO_{\PP F^\vee}(1)^{r-1}$ is strictly positive by our assumption.
Therefore $K_X^\vee$ is nef and big, i.e. $X$ is a weak Fano variety. Then the Mori cone of $X$ is rational polyhedral and the extremal rays are generated by rational curves by Theorem \ref{t2.6}. 

Finally we claim that $K_X^\vee$ intersects positively with all irreducible curves.
Let $C$ be an irreducible curve in $\PP E^\vee=\PP F^\vee$. If $q(C)$ is a point, then the degree of $\mathcal{O}_{\PP F^\vee}(r-1)|_C$ is positive because $\cO_{\PP F^\vee}(1)$ is ample on each fiber of $q : \PP E^\vee \to S$. If $q(C)$ is a curve, then the degree of $q^*(K_S^\vee \otimes \det E^\vee \otimes H^{r-1} )|_C$ is positive. Therefore we find that the degree of the line bundle $q^*(K_S^\vee \otimes \det E^\vee \otimes H^{r-1} ) \otimes \cO_{\PP F^\vee}(r-1)|_C$ is always positive. Since the Mori cone is polyhedral, this implies that $K_X^\vee$ is ample and $X$ is a Fano variety.
\end{proof}

\begin{rema}\label{r3.2}
In the proof of Theorem \ref{t1.3} in \cite[\S4.2]{KKLL}, we used $H=\cO_{\PP^{m+c}}(1)$ and chose sufficient large $c$ as written in \S\ref{s3.2}. However when the degrees of defining equations of $Y$ are large enough, then the above theorem tells us that we can choose larger $H$ and smaller $c$. This often gives a Fano host of smaller dimension as in the following example.
\end{rema}

\begin{exam}\label{e3.3}
Let $C$ be a non-hyperelliptic curve of genus 4. Then $C$ is the complete intersection of a quadric and a cubic in $\PP^3$, i.e. $C$ is the zero locus of a regular section $s$ of $E=\mathcal{O}_{\PP^3}(2) \oplus \mathcal{O}_{\PP^3}(3)$ over $S=\PP^3$ and let $F=\mathcal{O}_{\PP^3} \oplus \mathcal{O}_{\PP^3}(1)$ with $H=\cO_{\PP^3}(2)$. From the above theorem, we find that $X=w^{-1}(0)$ in Cayley's trick (cf. \S\ref{s3.1}) is a 3-dimensional Fano host of $C$ because $F$ is nef and $K_S^\vee\otimes \det E^\vee\otimes H^{r-1}=\cO_{\PP^3}(1)$ is ample. Note that if we insist on using $H=\cO_{\PP^3}(1)$ instead, we have to enlarge $\PP^3$ to $\PP^4$ and extend $\mathcal{O}_{\PP^3}(2) \oplus \mathcal{O}_{\PP^3}(3)$ to 
$\mathcal{O}_{\PP^4}(2) \oplus \mathcal{O}_{\PP^4}(3)\oplus \mathcal{O}_{\PP^4}(1)$, so that the Fano host is 5 dimensional.
\end{exam}

By Example \ref{e3.3}, we find that a non-hyperelliptic curve $C$ of genus 4 has Fano dimension at most 3. We will see below that indeed 3 is the Fano dimension of $C$.

\bigskip

We can extend the main theorem of \cite{KKLL} as follows.

\begin{theo}\label{t.wci}
Every quasi-smooth weighted complete intersection $Y$ in a weighted projective space is a coarse moduli space of a smooth Deligne-Mumford stack $\mathcal{Y}$ which has an orbifold Fano host. 
\end{theo}
\begin{proof}
Let $Y$ be a quasi-smooth weighted complete intersection is a weighted projective space $\PP(\bar{a}).$ We can embed $\PP(\bar{a})$ into $\PP(\bar{a},1,\cdots,1)$ and $Y$ is again a quasi-smooth weighted complete intersection in a weighted projective space $\PP(\bar{a},1,\cdots,1).$ Therefore we may assume that the dimension $n$ of $\PP(\bar{a})$ is large enough and almost all $a_i=1.$ Then $C^*_Y$ is a complete intersection in $\mathbb{C}^{n+1}-\{ 0 \}.$ We can regard $C_Y^*$ as a zero set of a section of the rank $c$ trivial vector bundle on $\mathbb{C}^{n+1}-\{ 0 \}.$ We can use Cayley's trick to construct $C_X^*$ as follows. 

\[\xymatrix{
& C^*_{X} \ar[r] & \mathbb{C}^{n+1} - \{ 0 \} \times \PP^{c-1} \ar[d] \\
C^*_{Y} \ar[rr] & & \mathbb{C}^{n+1}  - \{ 0 \}
}\]

Because $C_Y^*$ is defined by a $\mathbb{C}^*$-invariant section, we can naturally extend the $\mathbb{C}^*$-action to $\mathbb{C}^{n+1}-\{ 0 \} \times \PP^{c-1}$ and $C^*_X.$ Let $\mathcal{X}$ denotes the quotient stack $[C^*_X/\mathbb{C}^*]$ and $\mathcal{Y}$ denotes the quotient stack $[C^*_Y/\mathbb{C}^*].$ From the definition we see that $C_Y^*$ is smooth and $[C^*_Y/\mathbb{C}^*]$ is a smooth Deligne-Mumford stack whose coarse moduli space is $Y.$

From Orlov's theorem we see that $D^b([C^*_Y/\mathbb{C}^*])$ can be embedded into $D^b([C^*_X/\mathbb{C}^*]).$ Therefore we get the desired result from the following Lemma.
\end{proof}

\begin{lemm}
$\mathcal{X}$ is a smooth Deligne-Mumford stack whose coarse moduli space $X$ is a Fano variety.
\end{lemm}
\begin{proof}
Because $C^*_Y$ is smooth we see that $C^*_X$ is also smooth by local calculation. For every point of $C^*_X$ the stabilzer of the induced $\mathbb{C}^*$-action is a finite abelian group. Therefore $\mathcal{X}$ is a smooth Deligne-Mumford stack whose coarse moduli space is $X=C^*_X/\mathbb{C}^*.$ Therefore $X$ has only quotient singularities and $K_X$ is a $\QQ$-Cartier divisor.

Recall that $\PP(\bar{a})$ is a quotient of $\PP^n$ by $\mu_{\bar{a}}$-action where $\mu_{\bar{a}}$ acts on $\PP^n$ via 
$$ (\mu_0,\cdots,\mu_n) \cdot [z_0:\cdots:z_n]=[\epsilon _0^{\mu_0} \cdot z_0 : \cdots : \epsilon _n^{\mu_n} \cdot z_n] $$
where $\epsilon_i$ is a primitive $a_i$th root of unity (cf. \cite{BR, Dolgachev}).   

Then we can construct $C_{\widetilde{Y}}^*$ 

\[\xymatrix{
C^*_{\widetilde{Y}} \ar[d] \ar[r] & \mathbb{C}^{n+1}- \{0 \} \ar[d] \\
C^*_{Y} \ar[r] & \mathbb{C}^{n+1} - \{ 0 \}
}\]

using the lifting of the equations defining $C_Y.$

Note that $C_{\widetilde{Y}}^*$ is also a complete intersection in $\mathbb{C}^{n+1}-\{ 0 \}$ and we can also apply Cayley's trick to $C_{\widetilde{Y}}^*.$ By taking $\CC^*$-quotient we have the following diagram. 

\[\xymatrix{
 & & \widetilde{X} \ar[d] \ar[r] & \PP{E^\vee} \ar[d] \ar[lld] \\
\widetilde{Y} \ar[d] \ar[r] & \PP^n \ar[d] & X=\widetilde{X}/\mu_{\bar{a}} \ar[r] & \PP{E^\vee}/\mu_{\bar{a}} \ar[lld] \\
Y \ar[r] & \PP(\bar{a}) & &
}\]

In other words we can construct a ramified covering $\widetilde{X}$ of $X$ by applying Cayley's trick to $\PP^n$ which is also a ramified covering of $\PP(\bar{a}).$ 

When $n$ is large enough, we can see that $\widetilde{X}$ is a Fano variety from adjuction (cf. \cite[Proposition 5.73]{KM}) and codimensions of the fixed locus of the $G$-action is greater than equal to 2. Therefore we see that $X$ is also a Fano variety.
\end{proof}

\begin{rema}
Note that if $Y$ is singular quasi-smooth weighted complete intersection then $Y$ itself cannot have an orbifold Fano host. Therefore we should consider the derived category of smooth stack $\mathcal{Y}$ whose coarse moduli space is $Y$ instead of the derived category of $Y$ itself.
\end{rema}

However if $Y$ is smooth then we see that $Y$ itself has an orbifold Fano host.

\begin{coro}
Every smooth weighted complete intersection $Y$ in a weighted projective space has an orbifold Fano host.
\end{coro}
\begin{proof}
We see that $D^b(Y)$ is contained in $D^b(\mathcal{Y})$ from Lemma \ref{l.smoothcoarsemoduli} and $\mathcal{Y}$ has an orbifold Fano host from Theorem \ref{t.wci}. Therefore we see that $Y$ has an orbifold Fano host.
\end{proof}

\section{Fourier-Mukai transforms and an embeddability criterion}\label{s2.3}

In this section, we use the Fourier-Mukai transform to give a Hodge-theoretic criterion for the existence of a fully faithful functor $D^b(Y)\to D^b(X)$ for smooth projective varieties $X$ and $Y$. 

We first recall a fundamental result of Orlov that says all fully faithful exact functors are Fourier-Mukai. 
\begin{theo}\cite{Orlov1}\cite[Theorem 5.14]{Huy2}\label{t2.12} Let $X$ and $Y$ be two smooth projective varieties and let
$$ F : D^b(Y) \to D^b(X) $$
be a fully faithful exact functor. If $F$ admits right and left adjoint functors, then there exists an object $K \in D^b(Y \times X)$ unique up to isomorphism such that $F$ is isomorphic to the Fourier-Mukai transform $$\Phi_K={\pi_X}_*(\pi_Y^*(-)\otimes K):D^b(Y)\lra D^b(X)$$ 
where $\pi_X$ and $\pi_Y$ denote the projection morphisms from $X \times Y$ to $X$ and $Y$ respectively.
\end{theo}

\begin{rema}\label{r2.13}
The assumption that $F$ admits right and left adjoint functors can be dropped by a theorem of Bondal and van den Bergh \cite{BB}.
\end{rema}

The Fourier-Mukai kernel $K$ also defines the cohomological Fourier-Mukai transform $\Phi_K^H : H^*(Y,\QQ) \to H^*(X,\QQ)$.

\begin{defi}\cite{Huy2}\label{d2.14}
Let $K$ be a Fourier-Mukai kernel of $\Phi_K$. Then the cohomological Fourier-Mukai functor is the linear map
$$ \Phi_K^H : H^*(Y,\QQ) \to H^*(X,\QQ) $$
defined by
$$ \Phi_K^H(-)={\pi_X}_*(ch(K) \cdot \sqrt{td(X \times Y)} \cdot \pi_Y^*(-)).$$
\end{defi}

When a Fourier-Mukai functor gives an equivalence between derived categories of two smooth projective varieties, the induced cohomological Fourier-Mukai functor preserves the Hochschild homology groups.  

\begin{prop}\cite[Proposition 5.39]{Huy2}\label{p2.15}
If $\Phi_K : D^b(Y) \to D^b(X)$ is an equivalence, then the induced cohomological Fourier-Mukai transform $\Phi_K^H : H^*(Y,\QQ) \to H^*(X,\QQ)$ yields an isomorphism
$$ \bigoplus_{p-q=i} H^{p,q}(Y) \cong \bigoplus_{p-q=i} H^{p,q}(X)\quad \text{for all } i. $$
\end{prop}

We next generalize the above result to the case where $\Phi_K$ is a fully faithful functor.

\begin{prop}\cite[Corollary 1.22]{Huy2}\label{p2.16}
Let $F : \mathcal{A} \to \mathcal{B}$ be a fully faithful functor that admits a left adjoint $G \dashv F$ (resp.  right adjoint $F \dashv H$). Then the natural morphism
$$ G \circ F \to id_{\mathcal{A}} \quad \text{(resp.  }id_{\mathcal{A}} \to H \circ F ) $$  
is an isomorphism. 
\end{prop}

For Fourier-Mukai transforms, we always have left and right adjoint functors by a result of Mukai.

\begin{prop}\cite[Proposition 5.9]{Huy2}\label{p2.17}
A Fourier-Mukai transform admits left and right adjoint functors which are also Fourier-Mukai transforms.
\end{prop}

Now we can give a criterion for the existence of a fully faithful functor $D^b(Y)\to D^b(X)$.
\begin{prop}\label{p2.18} 
If a Fourier-Mukai transform $\Phi_K : D^b(Y) \to D^b(X)$ is fully faithful, then the induced cohomological Fourier-Mukai transform $\Phi_K^H : H^*(Y,\QQ) \to H^*(X,\QQ)$ yields an injective homomorphism
$$ \bigoplus_{p-q=i} H^{p,q}(Y) \subset \bigoplus_{p-q=i} H^{p,q}(X). $$
Hence, we have the inequality
$$\sum_{p-q=i} h^{p,q}(Y)\le \sum_{p-q=i} h^{p,q}(X)\quad \text{for all }  i.$$
\end{prop}
\begin{proof}
We will follow the arguments in \cite{Huy2}. There exists a right adjoint $\Phi_{K_R}$ of $\Phi_K$ and $\Phi_{K_R} \circ \Phi_K  \cong id \cong \Phi_{\mathcal{O}_{\Delta}}$ from the uniqueness of the Fourier-Mukai kernel. Then we get $\Phi_{K_R}^H \circ \Phi_K^H \cong \Phi_{\mathcal{O}_{\Delta}}^H \cong id$ (cf. \cite[Proposition 5.33]{Huy2}). Therefore $\Phi_K^H$ induces an inclusion $\Phi_K^H : H^*(Y,\CC) \to H^*(X,\CC)$ which satisfies
$$ \Phi_K^H(H^{p,q}(Y)) \subset \bigoplus_{r-s=p-q} H^{r,s}(X) $$
by the arguments in \cite[Proposition 5.39]{Huy2}.
\end{proof}

Proposition \ref{p2.18} can also be obtained from \cite[Theorem 7.6]{Kuz3} and later we found that Popa obtained a more general result for orbifold setting via similar argument in \cite{Popa}. Let us recall his result as follows.

\begin{prop}\cite[Proposition 2.2]{Popa}
Let $X$ and $Y$ be normal projective varieties having at worst quotient singularities. Let $\mathcal{X}$ and $\mathcal{Y}$ be smooth proper Deligne-Mumford stacks associated to $X$ and $Y,$ respectively. If a Fourier-Mukai transform $\Phi_{\mathcal{K}} : D^b(\mathcal{Y}) \to D^b(\mathcal{X})$ is fully faithful, then we have the following inequality
$$\sum_{p-q=i} h_{orb}^{p,q}(\mathcal{Y})\le \sum_{p-q=i} h_{orb}^{p,q}(\mathcal{X})\quad \text{for all }  i.$$
\end{prop}

A first consequence of Proposition \ref{p2.18} is the following lower bound.

\begin{coro}\label{c4.2}
Let $Y$ be an $n$-dimensional smooth projective variety with $h^{n,0}(Y) > 0$ for $n>0$. Then its Fano dimension is at least $n+2$.
\end{coro}
\begin{proof}
Suppose that there is a Fano variety $X$ of dimension at most $n+1$ and a fully faithful exact functor $ F : D^b(Y) \to D^b(X)$. By Proposition \ref{p2.18}, we have the inequality
$$0<h^{n,0}(Y)\le \bigoplus_{p-q=n}h^{p,q}(X).$$
Obviously the right hand side is zero unless $\dim X$ is $n$ or $n+1$.
By Lemma \ref{l2.7}, $h^{n,0}(Y)=0$. When $\dim X=n+1$, $h^{n+1,1}(X)=h^{n,0}(X)=0$. Hence the right hand side is always zero if $\dim X\le n+1$. This proves the proposition.
\end{proof}

When $\dim Y=1$ and $Y$ is not rational, $h^{1,0}(Y)>0$ and so we obtain the following.
\begin{coro}\label{c4.5}
The Fano dimension of a smooth projective curve which is not rational is at least 3.
\end{coro}
We will see below that the Fano dimension of a curve $Y$ is exactly 3 when the genus is 1 or 2 or when $Y$ is a general curve of genus 4. 

Combining Corollary \ref{c4.2} with the construction of Fano hosts in \cite{KKLL}, we can determine the Fano dimension of a general complete intersection Calabi-Yau variety.
\begin{prop}\label{p4.6}
Let $Y\subset \PP^{n+c}$ be a smooth projective complete intersection Calabi-Yau variety of dimension $n$ defined by the vanishing of homogeneous polynomials $f_1,\cdots,f_c$. Suppose $c\le 2$ or $Y$ is general in the sense that we can choose the defining polynomials such that the projective variety $S$ defined by the vanishing of $f_3,\cdots,f_c$ is smooth. Then the Fano dimension of $Y$ is precisely $n+2$. 
\end{prop}
\begin{proof}
By \cite[Proposition 3.6]{KKLL},  the Fano dimension of $Y$ is at most $n+2$. By Corollary \ref{c4.2}, the Fano dimension is at least $n+2$. This proves the proposition.
\end{proof}
For instance, the Fano dimension of an arbitrary quintic 3-fold is $5$ and the Fano host constructed in \cite{KKLL} is of minimal dimension.

\subsection{Complete intersection Calabi-Yau varieties}


Many algebraic varieties can be described as the zero loci of regular sections of vector bundles on Fano varieties. When they satisfy the assumption of Theorem \ref{t3.1}, we know that these varieties are Fano visitors and we can give upper bounds for their Fano dimensions.

The numerical conditions of Theorem \ref{t3.1} are particularly easy to check for Calabi-Yau varieties. Many Calabi-Yau varieties arise as complete intersections in homogeneous varieties. For these Calabi-Yau varieties, the construction in \S\ref{s3} gives the following proposition.

\begin{prop}\label{p7.1}
Let $Y$ be an $n$-dimensional Calabi-Yau variety. Suppose that there is an embedding of $Y$ into a smooth projective Fano variety $S$ of dimension $\geq n+2$ as the zero locus of a regular section of an ample vector bundle $E$ whose rank coincides with the codimension of Y. Let $m$ be the smallest dimensions of such an $S$. Then the Fano dimension of $Y$ is at most $2m-n-2$. In particular if $m = n+2$, then the Fano dimension of $Y$ is $n+2$.
\end{prop}

For example let us consider Calabi-Yau 3-folds and general type varieties in some Fano 4-folds. K\"{u}chle in \cite{Kuchle} classified Fano 4-folds of index 1 which are zero loci of vector bundles on homogeneous varieties. Then we can consider transversal linear sections of these Fano 4-folds. For instance, let us consider Fano 4-folds of (b7) type which are zero loci of global sections of $\mathcal{O}(1)^{\oplus 6}$ in $Gr(2,7)$. If the linear sections of these Fano 4-folds are smooth then it is easy to check that they are Fano visitors. The same argument works for transversal linear sections of many other Fano 4-folds described in \cite{Kuchle}. It seems interesting to study derived categories of linear sections of these Fano 4-folds. See \cite{Kuchle, Kuz5, Manivel} for the geometry of these varieties.


\subsection{Weighted complete intersection Calabi-Yau varieties}

There are lots of study of weighted complete intersection Calabi-Yau varieties in weighted projective spaces. See \cite{ABR, IF} for more details. We discuss a dimension bound for orbifold Fano hosts using the same idea which works for ordinary Fano hosts. 

\begin{lemm}\cite[Lemma 1.2]{Popa}
Let $\mathcal{X}$ be a smooth Deligne-Mumford stack associated to a normal projective variety $X.$ Let $\widetilde{X}$ be a resolution of singularities of $X.$ Then we have the following equality
$$ h_{orb}^{0,q}(\mathcal{X}) = h^{0,q}(\widetilde{X}) $$
for all $q.$
\end{lemm}

Then we can obtain the lower bound of dimensions of orbifold Fano hosts of Calabi-Yau varieties.

\begin{coro}
Let $Y$ be an $n$-dimensional normal projective Calabi-Yau variety with quotient singularities and $\mathcal{Y}$ be the associated orbifold. Suppose that $\mathcal{Y}$ has an orbifold Fano host $\mathcal{X}.$ Then the dimension of $\mathcal{X}$ is at least $n+2.$
\end{coro}
\begin{proof}
From the above Lemma we see that $h_{orb}^{0,q}(\mathcal{X})=h^{0,q}(\widetilde{X})=0$ for all $q > 0.$ Recall that the quotient singularities are rational singularities. Let $\widetilde{Y}$ be a resolution of singularities of $Y.$ Then we have $H^{0,n}(\widetilde{Y})=H^n(\widetilde{Y},\mathcal{O}_{\widetilde{Y}})=H^n(Y,\mathcal{O}_Y)=H^n(Y,K_Y)=H^0(Y,\mathcal{O}_Y)=\CC$ and hence $h_{orb}^{0,n}(\mathcal{Y})=h^{0,n}(\widetilde{Y})=1.$ If the dimension of $\mathcal{X}$ is $n$ then we have the following contradiction.
$$ 1 \leq \sum_{p-q=n} h_{orb}^{p,q}(\mathcal{Y})\le \sum_{p-q=n} h_{orb}^{p,q}(\mathcal{X}) = h_{orb}^{n,0}(\mathcal{X}) = 0. $$
If the dimension of $\mathcal{X}$ is $n+1$ then we have the following contradiction.
$$ 1 \leq \sum_{p-q=n} h_{orb}^{p,q}(\mathcal{Y})\le \sum_{p-q=n} h_{orb}^{p,q}(\mathcal{X}) = h_{orb}^{n+1,1}(\mathcal{X}) + h_{orb}^{n,0}(\mathcal{X}) = 2 h_{orb}^{n,0}(\mathcal{X}) = 0. $$
Therefore we obtain the desired result.
\end{proof}

See \cite[Corollary 6.7]{Popa} for more details about the above computation. Finally we obtain the following. 

\begin{coro}
Let $Y$ be an $n$-dimensional general weighted complete intersection Calabi-Yau variety in a weighted projective space and let $\mathcal{Y}$ be the associated Deligne-Mumford stack. Then there is $(n+2)$-dimensional Fano orbifold $\mathcal{X}$ such that $D^b(\mathcal{Y})$ is contained in $D^b(\mathcal{X})$ and $n+2$ is the minimum dimensions of such Fano orbifolds.  
\end{coro}
\begin{proof}
The proof of the existence of $(n+2)$-dimensional orbifold Fano host is similar to the proof of the same statement of smooth projective Calabi-Yau varieties. The statement that $n+2$ is the minimum possible dimensions of Fano orbifolds follows from the above Corollary.
\end{proof}

\section{Curves}\label{s5}

In this section we search for Fano visitors among smooth projective curves. Curves in this section mean smooth projective curves.

\subsection{Hyperelliptic curves}
Bondal and Orlov proved that every hyperelliptic curve is a Fano visitor.

\begin{theo}\cite{BO}\label{t5.1}
Let $C$ be a hyperelliptic curve of genus $g$. Then there are two quadric hypersurfaces in $\PP^{2g+1}$ whose intersection is a Fano host of $C$. 
\end{theo}

\begin{coro}\label{c5.2}
A hyperelliptic curve $C$ of genus $g$ is a Fano visitor whose Fano dimension is at most $2g-1$.
\end{coro}
This corollary indicates that the Fano dimension of a curve of genus $g$ might increase as $g$ increases.
Indeed the Fano dimension of a curve of genus $g$ may grow arbitrarily large as $g$ increases.
\begin{prop}\label{p5.3}
Let $\mathrm{fd}(g)$ be the minimum among the Fano dimensions of curves of genus $g$. Then 
$\lim_{g\to \infty} \mathrm{fd}(g)=\infty$.
\end{prop}
\begin{proof}
For any natural number $n$, there are only finitely many deformation equivalence classes of Fano varieties of dimension $n$. Therefore there are only finitely many possible values of $\sum_{i-j=1} h^{i,j}(X)$ for $n$-dimensional Fano varieties $X$. When the genus $g=h^{1,0}(C)$ of a curve $C$ is greater than all these possible values, there can be no $n$-dimensional Fano host of $C$. Therefore for any integer $n>0$ there is an integer $g_0$ such that any curve of genus $g\ge g_0$ has Fano dimension greater than $n$. This proves the proposition.
\end{proof}

We can also prove that every hyperelliptic curve has an orbifold Fano host.

\begin{coro}
Let $C$ be a hyperelliptic curve of genus $g.$ Then $C$ has an orbifold Fano host.
\end{coro}
\begin{proof}
We can see $C$ as a complete intersection in a weighted projective space $\PP(1,1,g+1)$ (cf. \cite{Dolgachev}). Therefore it follows immediately from Theorem \ref{t.wci}. 
\end{proof}

\subsection{Low genus curves}\label{s5.1}
In this subsection we prove that all curves $C$ of genus $g\le 4$ are Fano visitors. 
If $g=0$,  $C=\PP^1$ itself is a Fano variety.
If $g=1$, $C\subset \PP^2$ is a complete intersection Calabi-Yau variety of codimension 1 and hence its Fano dimension is 3 by Proposition \ref{p4.6}.
\begin{coro}\cite[Example 5.3]{KKLL}\label{c5.5}
Every elliptic curve is a Fano visitor and its Fano dimension is 3.
\end{coro}
If $g=2$, $C$ is a hyperelliptic curve and hence the Fano dimension is at most $3$ by Corollary \ref{c5.2}. By Corollary \ref{c4.5}, the Fano dimension of $C$ is at least 3. So we proved the following. 
\begin{coro}\label{c5.6}
Every curve of genus 2 is a Fano visitor with Fano dimension 3.
\end{coro}

If $g=3$, it is well known that $C$ is either a plane quartic or a hyperelliptic curve. In the former case, we use the construction in \S\ref{s3.2} with $l=1$, $m=2$, $a_1=4$, $c=2$ to obtain a Fano host $X$ of dimension $5$. In the latter case, Theorem \ref{t5.1} gives a Fano host of dimension 5. So we obtain the following. 
\begin{coro}\label{c5.7}
Every curve of genus 3 is a Fano visitor and the Fano dimension is at most 5.
\end{coro}

If $g=4$, it is well known that $C$ is either the complete intersection of a quadric and a cubic in $\PP^3$ or a hyperelliptic curve. In the former case, the Fano dimension is exactly 3 by Example \ref{e3.3}. In the latter case, the Fano dimension is at most 7. 
\begin{coro}\label{c5.8}
Every curve $C$ of genus 4 is a Fano visitor with Fano dimension at most 7. If $C$ is non-hyperelliptic, then its Fano dimension is 3. 
\end{coro}

\subsection{General curves of genus $\leq 9$}\label{s5.2}
In this subsection, we use Mukai's description of general curves $C$ of genus $g\le 9$ as complete intersections in homogeneous varieties and prove that they are Fano visitors.


A general curve $C$ of genus 5 has canonical embedding into $\PP^4$ whose image is the intersection of three general quadrics. 
Let $S$ be one of the quadric hypersurfaces and let $s$ be the section of $E=\cO_{\PP^4}(2)^{\oplus 2}|_S$ defined by the remaining two quadrics, so that $C=s^{-1}(0)$. Then $\PP E^\vee\cong S\times \PP^1$ is a Fano variety and hence the Mori cone of $\PP E^\vee$ is rational polyhedral. Let $H=\cO_{\PP^4}(2)|_S$. Then $F=E\otimes H^{-1}=\cO_{\PP^4}^{\oplus 2}$ is nef and $K_S^\vee\otimes \det E^\vee\otimes H=\cO_{\PP^4}(1)|_S$ is ample.
Therefore we obtain the following from Theorem \ref{t3.1}. 
\begin{coro}\label{c5.9}
A general curve of genus 5 is a Fano visitor and its Fano dimension is 3.
\end{coro}

For higher genus curves, we recall some results of Mukai.

\begin{defi}\label{d5.10}
A curve $C$ has a $g_d^r$ if there is a line bundle $L$ on $C$ with  $\deg L=d$ and $h^0(C,L) \geq r+1$.
\end{defi}

\begin{prop}\cite[Proposition 1.9]{Mukai4}\label{p5.11}
The anticanonical line bundle of the Grassmannian $Gr(k,n)$ is $\mathcal{O}(n)$ where $\mathcal{O}(1)$ is the very ample line bundle which gives the Pl\"{u}cker embedding. 
\end{prop}

Mukai proved that a general curve of genus 6 can be embedded into $Gr(2,5)$ as a complete intersection.

\begin{theo}\cite{Mukai4}\label{t5.12}
A curve $C$ of genus 6 is the complete intersection of $Gr(2,5) \subset \PP^9$ and a 4-dimensional quadric in $\PP^5\subset\PP^9$ if $C$ is not bi-elliptic and has no $g_3^1$ or $g_5^2$.
\end{theo}

\begin{prop}\cite[Proposition 2.1]{Mukai5}\label{p5.13}
The anticanonical line bundle of the 10-dimensional orthogonal Grassmannian variety $X_{12}^{10}$ is $\mathcal{O}(8)$ where $\mathcal{O}(1)$ is the very ample line bundle which gives the Pl\"{u}cker embedding. 
\end{prop}

Mukai proved that a general curve of genus 7 can be embedded in $X_{12}^{10}$ as a complete intersection.

\begin{theo}\cite{Mukai5}\label{t5.14}
A curve $C$ of genus 7 is a transversal linear section of $X_{12}^{10} 
\subset \PP^{15}$ if and only if $C$ has no $g_4^1$.
\end{theo}

Mukai proved that a generic curve of genus 8 can be embedded in $Gr(2,6)$ as a complete intersection.

\begin{theo}\cite{Mukai4}\label{t5.15}
A curve $C$ of genus 8 is a transversal linear section of $G(2,6) \subset \PP^{14}$ if and only if $C$ has no $g_7^2$.
\end{theo}

\begin{prop}\cite[Proposition 2.3]{Mukai7}\label{p5.16}
The symplectic Grassmannian $SpGr(n,2n)$ is a smooth projective variety of dimension $n(n+1)/2$ whose anticanonical line bundle is $\mathcal{O}(n+1)$ where $\mathcal{O}(1)$ is the very ample line bundle which gives the Pl\"{u}cker embedding. 
\end{prop}

Mukai proved that a general curve of genus 9 can be embedded in $SpGr(3,6)$ as a complete intersection.

\begin{theo}\cite{Mukai7}\label{t5.17}
A curve $C$ of genus 9 is a transversal linear section of $SpGr(3,6) \subset \PP^{13}$ if and only if $C$ has no $g_5^1$.
\end{theo}

\begin{coro}\label{c5.18}
A general curve of genus $g$ with $1 \leq g \leq 9$ is a complete intersection in a homogeneous variety.
\end{coro}

\begin{theo}\label{t5.19}
General curves of genus $g \leq 9$ are Fano visitors.
\end{theo}
\begin{proof}
We already proved that general curves of genus $\leq 5$ are Fano visitors by using their canonical embeddings. 
Let $C \subset Z \subset \PP^N$ be a curve of genus $6 \leq g \leq 9$ which is a complete intersection in a homogeneous variety $Z$ embedded in $\PP^N$ via the Pl\"{u}cker embedding. From the adjunction formula we see that $K_C \cong \mathcal{O}_{\PP^N}(1)|_C$. In each case, we can find varieties $C \subset S \subset Z \subset \PP^N$ where $S$ is a 4-dimensional complete intersection in $Z$ and $C$ is the zero locus of a section of a rank 3 vector bundle on $S$. We then find that the variety $S$ and the rank 3 vector bundle satisfy the assumptions of Theorem \ref{t3.1}. Therefore $C$ is a Fano visitor. Moreover we see that the Fano dimensions of general curves of genus $6 \leq g \leq 9$ are at most 5.
\end{proof}

Theorem \ref{t3.1} enables us to provide many more examples of curves of genus $\geq 10$ which are Fano visitors.

\begin{rema}\label{Nar}
After we finished writing this paper, we received a manuscript from M. S. Narasimhan \cite{Narasimhan} in which he proves that all curves of genus at least 6 are Fano visitors. He also proves that all non-hyperelliptic curves of genus 3, 4 or 5 are Fano visitors. Combined with Theorem \ref{t5.1}, these results prove that all curves are Fano visitors.

It is well known that the moduli space of rank 2 stable vector bundles over a curve with fixed odd determinant is Fano. Narasimhan proves that the Fourier-Mukai transform defined by the universal bundle is fully faithful. Recently, Fonarev and Kuznetsov obtained similar results for generic curves, especially for all hyperelliptic curves via different method  (cf. \cite{FK}). It follows that the Fano dimension of an arbitrary curve of genus $g\ge 2$ is at most $3g-3.$ But our discussion above for curves of low genus indicates that this upper bound is far from being optimal.
\end{rema}

We end this section with the following question.
\begin{ques} 
What is the stratification on the moduli space $M_g$ of smooth curves of genus $g$, defined by the Fano dimension?  
\end{ques}
It will be interesting to compare the stratification by Fano dimension with other known stratifications on $M_g$.

\subsection{Jacobians of curves}

Let $C$ be a curve and $J(C)$ be the Jacobian of $C.$ It is a classical topic in algebraic geomety to study interactions between $C$ and $J(C).$ Because every curve is a Fano visitor we can prove that every Jacobian of a curve has orbifold Fano hosts. 

\begin{prop}
Let $C$ be a curve and $J(C)$ be its Jacobian. Then $J(C)$ has orbifold Fano hosts.
\end{prop}
\begin{proof}
Let $F$ be a smooth projective Fano host of $C.$ It is well-known that there is a surjection $\phi^{(n)} : C^{(n)} \to J(C)$ such that $R\phi^{(n)}_*\mathcal{O}_{C^{(n)}}=\mathcal{O}_{J(C)}$ for $n > 2g-2$ and we see that this surjection induces a fully faithful functor $D^b(J(C)) \to D^b(C^{(n)}).$ From the Lemma \ref{l.smoothcoarsemoduli} we see that there is a fully faithful functor $D^b(C^{(n)}) \to D^b([C^n/S_n]).$ Again from the Theorem \ref{t.KSP} we see that there is a fully faithful functor $D^b([C^n/S_n]) \to D^b([F^n/S_n]).$ Therefore we see that for every Jacobian of curve $J(C)$ there is an orbifold Fano host $[F^n/S_n].$
\end{proof}

See \cite{PV} for more details about the semiorthogonal decomposition of $D^b([C^n/S_n]).$

\section{Surfaces}\label{s6}

In this section we discuss the Fano visitor problem for surfaces. 
Surfaces in this section always mean normal projective surfaces.
Because Grothendieck groups of many interesting surfaces contain finite abelian groups as direct summands we also consider orbifold Fano hosts of them. 
Unfortunately, we do not have a uniform way to construct (orbifold) Fano hosts of surfaces so we raise more questions than give results. 
Let $Y$ be a surface and $\kappa$ denote its Kodaira dimension.

First, one can ask whether it is enough to consider the Fano visitor problem for minimal surfaces only. 

\begin{ques}\label{q6.1}
Let $Y$ be a smooth projective surface and $\tilde{Y}$ denote the blowup of $Y$ at a point. Is $\tilde{Y}$ a Fano visitor if $Y$ is a Fano visitor? More generally, is a variety birational to a Fano visitor a Fano visitor?
\end{ques}

Many algebraic surfaces have fibration structures. Therefore the following question makes sense.

\begin{ques}\label{q6.2}
Suppose that $Y$ is a fiber bundle over a variety $B$ with fiber $F$. Is $Y$ a Fano visitor if $B$ and $F$ are Fano visitors?
\end{ques}

When the fibration is trivial, the answer to this question is a direct consequence of  the following.

\begin{prop}\cite[Corollary 7.4]{Huy2}\label{p6.3}
Let $\Phi_K :D^b(A) \to D^b(X)$ and $\Phi_{K'} :D^b(A') \to D^b(X')$ be two fully faithful Fourier-Mukai transforms. Then
$$ \Phi_{K \boxtimes K'} : D^b( A \times A' ) \to D^b( X \times X') $$
is also fully faithful.
\end{prop}

Therefore we have the following.
\begin{coro}\label{c6.4}
If $B$ and $F$ are Fano visitors then $B \times F$ is a Fano visitor.
\end{coro}

\subsection{$\kappa = -\infty$ case}\label{s6.1}

If the answer to Question \ref{q6.1} is yes, then we may assume $Y$ is either $\PP^2$, a Hirzebruch surface or a ruled surface. If the answer to Question \ref{q6.2} is also yes, then the all surfaces with $\kappa = -\infty$ are Fano visitors by Remark \ref{Nar}.

Now let us provide several examples of ruled surfaces having Fano hosts.

\begin{prop}\label{t.ruled}
Let $C=s^{-1}(0)$ be a smooth projective variety which is defined by a regular section of a rank $r \geq 2$ vector bundle $E$ on $S$ and let $F$ be a rank 2 vector bundle on $S.$ Suppose that there are line bundles $H_1$ and $H_2$ such that $q^*(E \otimes H_1^{\vee}),$ $F \otimes H_2^{\vee},$ $K_S^\vee \otimes \det E^\vee \otimes det F^\vee \otimes H_1^{r-1} \otimes H_2^2$ are nef vector bundles and at least one of them is ample. Then $\PP(F^\vee|_C)$ is a Fano visitor.
\end{prop}
\begin{proof}
It is obvious that $\PP(F^\vee|_C)$ is a complete intersection of a regular section of $q^*E$ in $\PP(F^\vee).$ We can use Cayley's trick to construct Fano host $X$ of $\PP(F^\vee|_C)$ because it is a complete intersection of a regular section $q^*E$ in $\PP(F^\vee).$ Then we have the following diagram

\[\xymatrix{
 & X \ar[r] & \PP(q^*E^\vee) \ar[d]^p \\
\PP(F^\vee|_C) \ar[d] \ar[rr] & & \PP F^\vee \ar[d]^q \\
C \ar[rr] & & S
}\]

and
$$ K^\vee_{\PP F^\vee} \cong q^*(K_S^\vee \otimes det F^\vee) \otimes \mathcal{O}_{\PP F^\vee}(2) $$
and 
$$ K^\vee_{\PP(q^*E^\vee)} \cong p^*(K_{\PP F^\vee}^\vee \otimes det(q^*E^\vee)) \otimes \mathcal{O}_{\PP(q^*E^\vee)}(r) $$
$$ \cong p^*(q^*(K_S^\vee \otimes det F^\vee \otimes det E^\vee) \otimes \mathcal{O}_{\PP F^\vee}(2)) \otimes \mathcal{O}_{\PP(q^*E^\vee)}(r). $$
From the construction we see that $D^b(\PP(F^\vee|_C))$ can be embedded into $D^b(X)$ and 
$$ K^\vee_X \cong p^*(q^*(K_S^\vee \otimes det F^\vee \otimes det E^\vee) \otimes \mathcal{O}_{\PP F^\vee}(2)) \otimes \mathcal{O}_{\PP(q^*E^\vee)}(r-1). $$
Using the same argument of proof of Theorem \ref{t3.1}, we can see that $X$ is Fano. 
\end{proof}

Therefore we get the following results.

\begin{coro}
All Hirzebruch surfaces are Fano visitors.
\end{coro}
\begin{proof}
Let $Y \cong \PP(\mathcal{O}_{\PP^1} \oplus \mathcal{O}_{\PP^1}(a))$ be a Hirzebruch surface. We can embed $\PP^1$ into $\PP^{r+1}.$ Let $S=\PP^{r+1},$ $F=\mathcal{O}_{\PP^{r+1}} \oplus \mathcal{O}_{\PP^{r+1}}(a),$ $E=\mathcal{O}(1)^{\oplus r},$ and $E'=\mathcal{O}^{\oplus r}.$ From the same construction and notation as the above Proposition \ref{t.ruled} we have
$$ K^\vee_{\PP(q^*E^\vee)} \cong p^*(K_{\PP F^\vee}^\vee \otimes det(q^*E^\vee)) \otimes \mathcal{O}_{\PP(q^*E^\vee)}(r) $$
$$ \cong p^*(q^*(\mathcal{O}(r+2) \otimes \mathcal{O}(-a) \otimes \mathcal{O}(-r)) \otimes \mathcal{O}_{\PP F^\vee}(2)) \otimes \mathcal{O}_{\PP(q^*E^\vee)}(r) $$
$$ \cong p^*(q^*(\mathcal{O}(r+2-a)) \otimes \mathcal{O}_{\PP F^\vee}(2)) \otimes \mathcal{O}_{\PP(q^*E'^\vee)}(r). $$
When $r$ is sufficiently large, the above construction gives Fano host of $Y.$
\end{proof}

By the same proof we obtain the following.

\begin{coro}
Let $C$ be a curve which is a complete intersection in a projective space. Let $F$ be a vector bundle of rank 2 on the projective space which is a direct sum of two line bundles. Then $\PP(F^\vee|_C)$ is a Fano visitor.  
\end{coro}

 


\subsection{$\kappa = 0$ case}\label{s6.2}

\subsubsection{Abelian surfaces}

An Abelian surface which is the product of two elliptic curves is a Fano visitor by Corollary \ref{c6.4}. 

\begin{exam}
Let $Y = E \times E'$ be the product of two elliptic curves. Then $Y$ is a Fano visitor.
\end{exam}

\begin{exam}
Note that the Jacobian of any genus 2 curve always has an orbifold Fano host.
\end{exam}

\subsubsection{K3 surfaces}
 
The following is a consequence of Theorem \ref{t3.1} for K3 surfaces. 
 
\begin{coro}\label{c6.6}
Let $Y$ be a $K3$ surface which is the zero locus of a section of an ample vector bundle $E$ of rank $r$ on a Fano variety $S$ of dimension of $r+2$ where $r \geq 2$. Then $Y$ is a Fano visitor. The Fano dimension of $Y$ is at most $2r$. 
\end{coro}

\begin{exam}\label{e6.7}
Let $V$ be a Fano 3-fold and let $Y$ be a smooth divisor in $|K_V^\vee|$ which is a $K3$ surface by adjunction. When $V$ is the zero locus of a regular section of an ample vector bundle on another Fano manifold $W$ and the line bundle $K_V$ is the restriction of an ample line bundle on $W$, we find that $Y$ is a Fano visitor by Theorem \ref{t3.1}. For example, general $K3$ surfaces of genus $6 \leq g \leq 10$ satisfy these conditions (cf. \cite{Mukai1}). Therefore general $K3$ surfaces of genus $6 \leq g \leq 10$ are Fano visitors and their Fano dimensions are 4.
\end{exam}

The above result can be used to find orbifold Fano hosts of holomorphic symplectic varieties. Recall that the Hilbert schemes of points on K3 surfaces are holomorphic symplectic varieties.

\begin{coro}
Let $Y$ be a K3 surface and $X$ be a Fano host of $Y.$ Then $[X^n/S_n]$ is an orbifold Fano host of $Y^{[n]}.$
\end{coro}
\begin{proof}
From the Bridgeland-King-Reid-Haiman correspondence (cf. Theorem \ref{t.BKRH}) we see that $D^b(Y^{[n]}) \simeq D^b([Y^n/S_n]).$ Then from Theorem \ref{t.KSP} we see that $D^b([Y^n/S_n])$ can be embedded into $D^b([X^n/S_n]).$ Therefore we get the desired result.
\end{proof}

Now let us consider Kummer surfaces. Consider an Abelian surface $A$ having Fano host and consider an involution $\sigma$ on $A$ which send $x \mapsto -x$ with respect to the group structure on $A.$ Then $\sigma$ has 16 fixed points and the minimal resolution of $A/\sigma$ is a K3 surface $S.$ We call $S$ a Kummer surface. One can prove that if $A$ has a Fano host $F$ such that $\sigma$ extends to $F$ and the Fourier-Mukai kernel of the embedding is $\sigma$-invariant, then $S$ has an orbifold Fano host. Let us give such examples as follows.

\begin{prop}
Let $E_1, E_2$ be elliptic curves, $A=E_1 \times E_2$ be an Abelian surface and let $S$ be the associated Kummer surface. Then $S$ has an orbifold Fano host.
\end{prop}
\begin{proof}
In this case $\sigma$ is induced by two involutions $\sigma_1, \sigma_2$ on $E_1,E_2$ respectively. For each $i,$ the 2-torsion points of $E_i$ form a $\sigma_i$-invariant divisor of degree 4 on $E_i$ and gives an embedding $E_i \to \PP^3.$ Then $\sigma_i$-action extends to $\PP^3$ and it also extends to $F_i=Bl_{E_i} \PP^3.$ Note that there exists embedding $\Phi_{K_i} : D^b(E_i) \to D^b(F_i)$ for $i=1,2.$ Then we have a fully faithful functor $\Phi_{K_1 \boxtimes K_2} : D^b(E_1 \times E_2) \to D^b(F_1 \times F_2).$ Let us consider the diagonal $\ZZ_2$-actions on $E_1 \times E_2,$ $F_1 \times F_2$ and $E_1 \times E_2 \times F_1 \times F_2.$ It is easy to see that $K_1 \boxtimes K_2$ is $\ZZ_2$-invariant hence $\ZZ_2$-linearized. By McKay correspondence (cf. \cite{BKR}) we see that $D^b(S) \cong D^b([(E_1 \times E_2)/\ZZ_2])$ can be embedded into $D^b([(F_1 \times F_2)/\ZZ_2])$. Therefore $S$ has an orbifold Fano host.
\end{proof}

Reid constructed 95 families of orbifold K3 surfaces as complete intersections in weighted projective spaces (cf. \cite{ABR}). We can apply our method to prove that they have orbifold Fano hosts.

\begin{coro}
Let $Y$ be one of the Reid's orbifold K3 surfaces and let $\mathcal{Y}$ be a smooth DM stack associated to $Y.$ Then $\mathcal{Y}$ has an orbifold Fano host.
\end{coro}
\begin{proof}
Reid's 95 families of orbifold K3 surfaces are complete intersections in weighted projective spaces. Therefore it follows immediately from Theorem \ref{t.wci}. 
\end{proof}

\subsubsection{Enriques surfaces}

We will review a construction of an Enriques surface described in \cite[Example 8.18]{Beauville}. Let $Q_1(z_0,z_1,z_2)+Q_1'(z_3,z_4,z_5),$ $Q_2(z_0,z_1,z_2)+Q_2'(z_3,z_4,z_5),$ $Q_3(z_0,z_1,z_2)+Q_3'(z_3,z_4,z_5)$ be three quadric forms with variables $z_0,\cdots,z_5$ and let $Y$ be a K3 which is intersection of three quadrics hypersurfaces defined by these three quadric forms in $\PP^5.$ Then $Y$ is a smooth K3 surface if we choose $Q_i,Q_i'$ generically. Let $\sigma$ be an involution on $\PP^5$ defined as follows. 
$$ \sigma \cdot [z_0:z_1:z_2:z_3:z_4:z_5] = [z_0:z_1:z_2:-z_3:-z_4:-z_5]. $$
Then $\sigma$ induces a fixed point free involution $\sigma$ on a K3 surface $Y$ when we choose $Q_i,Q_i'$ generically. It is known that the generic Enriques surface can be obtained in the above construction. See \cite[Example 8.18]{Beauville} for more details.

\begin{prop}
A generic Enriques surface has an orbifold Fano host.
\end{prop}
\begin{proof}
Let $S$ be an Enriques surface obtained as the quotient of a K3 surface $Y$ which is constructed as above and let $X$ be the Fano host of $Y$ constructed by Cayley's trick. Note that $\sigma$ induces an involution $\sigma_X$ on $X$ because $Y$ is defined by $\sigma$-invariant sections. Let $S$ be the Enriques surface whose double cover is $Y$ and we see that $D^b(Y/\langle \sigma \rangle)=D^b([Y/\langle \sigma \rangle]) \hookrightarrow D^b([X/ \langle \sigma_X \rangle]).$ Therefore the Enriques surface $S=Y/\langle \sigma \rangle$ has an orbifold Fano host.
\end{proof}

\subsubsection{Bielliptic surfaces}

Recall that a bielliptic surface is the quotient of product of two elliptic curves by finite abelian group. They were classified by Bagnera and de Franchis. Let us recall their classification as follows. See \cite{Beauville, CCS} for more details.

\begin{theo}\cite{Beauville}
Let $S=(E_1 \times E_2)/G$ be a bielliptic surface. Then $G$ is a finite group of translations of $E_1$ and $G$ is acting on $E_2$ as follows: \\
(1) $G=\ZZ_2$ and $G$ acts on $E_2$ by symmetry. \\
(2) $G=\ZZ_2^2$ and $G$ acts on $E_2$ by $x \mapsto -x$ and $x \mapsto x + \epsilon$ where $\epsilon$ is a nontrivial two torsion element of $E_2.$ \\
(3) $G=\ZZ_4$ and $G$ acts on $E_2$ by $x \mapsto ix.$ \\
(4) $G=\ZZ_4 \oplus \ZZ_2$ and $G$ acts on $E_2$ by $x \mapsto ix$ and $x \mapsto x+\frac{1+i}{2}.$ \\
(5) $G=\ZZ_3$ and $G$ acts on $E_2$ by $x \mapsto \rho x$ where $\rho$ is a primitive 3rd root of unity. \\
(6) $G=\ZZ_3^2$ and $G$ acts on $E_2$ by $x \mapsto \rho x$ and $x \mapsto x+\frac{1-\rho}{3}$ \\
(7) $G=\ZZ_6$ and $G$ acts on $E_2$ by $x \mapsto - \rho x.$
\end{theo}

It is easy to see the following simple Lemma holds.

\begin{lemm}\label{l.equiv.emb.ell.curve}
Let $E$ be an elliptic curve and let $G$ be a finite group acting on $E.$ Then there is a $G$-invariant ample divisor which induces an embedding $E \to \PP^n$ and the $G$-action on $E$ extends to $\PP^n.$ Moreover if $2 \leq |G| \leq 4$ then $E$ is a zero locus of $G$-invariant regular section of vector bundle in $\PP^n$
\end{lemm}
\begin{proof}
From the assumption, there is a $G$-invariant ample divisor $D$ on $E.$ Then there is a $G$-equivariant line bundle $\mathcal{O}(D)=L$ of degree $|G|$ on $E.$ If $|G|=3,4,$ then $L$ induces an embedding of $E$ to $\PP^{|G|-1}.$ Because $L$ is a $G$-equivariant line bundle, $H^0(E,L)$ is a $G$-module structure. Therefore the $G$-action on $E$ extends to $\PP^{|G|-1}.$ Because $G$ is an abelian group, $H^0(E,L)$ is a direct sum of 1-dimensional representations of $G.$ If $|G|=2,$ then we can use $L^{\otimes 2}$ and obtain similar result. Therefore we get the desired result from the information about the syzygies of elliptic curves in projective spaces of low dimensions (cf. \cite{ACGH, Eisenbud}).
\end{proof}

Then we have the following result.

\begin{prop}
A bielliptic surface $S=(E_1 \times E_2)/G$ where $|G| \leq 4$ has an orbifold Fano host.
\end{prop}
\begin{proof}
Let $S=(E_1 \times E_2)/G$ be a bielliptic surface. From the above Lemma \ref{l.equiv.emb.ell.curve}, we can construct two Fano hosts $F_i$ of $E_i$ with embedding $\Phi_{K_i} : D^b(E_i) \to D^b(F_i)$ for $i=1,2.$ For example, let $F_i$ be the blowup of $E_i$ in $\PP^3.$ Then we have a fully faithful functor $\Phi_{K_1 \boxtimes K_2} : D^b(E_1 \times E_2) \to D^b(F_1 \times F_2).$ It is easy to check that $K_1 \boxtimes K_2$ can be $G$-linearized. Therefore we have the desired embedding $\Phi_{K_1 \boxtimes K_2} : D^b([E_1 \times E_2/G]) \to D^b([F_1 \times F_2/G])$ and $(F_1 \times F_2)/G$ is a Fano variety.
\end{proof}

\subsection{$\kappa = 1$ case}\label{s6.3}

Because every minimal surfaces with $\kappa = 1$ is an elliptic surface, it is natural to ask the following.

\begin{ques}\label{q6.9}
Let $Y$ be an elliptic surface over a curve $C$. Is $Y$ a Fano visitor if $C$ is a Fano visitor?
\end{ques}

Again we do not know the answer to this question unless $Y\to C$ is a trivial fibration.

Let us discuss one more type of examples of surfaces with $\kappa=1.$

\begin{exam}
Let $E$ be an elliptic curve and $C$ be a curve of genus $g \geq 2.$ Let $G$ be a finite group of translations of $E$ and suppose that $G$ is acting on $C.$ Consider $E \times C$ and the diagonal $G$-action on it. Let $Y=(E \times C)/G.$ Because the diagonal action on $E \times C$ is free, $Y$ is a surface with $\kappa=1.$
\end{exam}

In order to construct an orbifold Fano host of $Y,$ we will use moduli space of rank 2 vector bundles on $C$ which turns out to be a Fano host of $C$ (cf. \cite{FK, Narasimhan}). Let us prove that the universal vector bundle of the moduli space of rank $2,$ $G$-invariant fixed odd degree line bundle is a $G$-invariant vector bundle where $G$ is a finite group acting on the curve.

\begin{lemm}\label{l.invuniv}
Let $C$ be a curve with a $G$-action where $G$ is a finite group. Suppose that $C$ has a $G$-invariant line bundle $\xi$ of odd degree. Then there is a natural action on the moduli space $M$ of rank 2 vector bundles on $C$ and the universal bundle is a $G$-invariant vector bundle with respect to the diagonal action.
\end{lemm}
\begin{proof}
Let $E \in M$ and $g \in G.$ We can define the $G$-action on $M$ by $g \cdot E = (g^{-1})^*E.$ Therefore we have a diagonal action on $C \times M.$ Let $U$ be the universal vector bundle on $C \times M.$ Note that $g^*U$ is a flat family of rank 2 vector bundles on $C \times M.$ Therefore $g^*U$ induces an isomorphism $\phi_g :M \to M$ such that $g^*U \cong (Id_C \times \phi_g)^*U.$ From the definition of the action one can check that $\phi_g$ is an identity morphism from $M$ to $M.$ Therefore $U$ is a $G$-invariant vector bundle on $C \times M.$
\end{proof}

We proved that $U$ is a $G$-invariant vector bundle if $\xi$ is $G$-invariant line bundle of odd degree. However it does not mean that $U$ is a $G$-equivariant vector bundle. Indeed if $\xi$ is not $G$-equivariant line bundle then $U$ is not $G$-equivariant vector bundle. We have the following numerical condition when $U$ being $G$-invariant imply $U$ being $G$-equivariant.

\begin{lemm}\label{l.equivuniv}
Let $G$ be a finite group acting on a variety $X$ and let $U$ be a $G$-invariant rank 2 simle vector bundle whose determinant $\xi$ is a $G$-equivariant line bundle. Suppose that $gcd(2,|H^2(G,\mathbb{C}^*)|)=1$ then $U$ is a $G$-equivariant vector bundle.
\end{lemm}
\begin{proof}
Because $U$ is a $G$-invariant vector bundle we have an isomorphism $\theta_g : g^*U \to U$ for each $g \in G.$ Because $U$ is simple, we have an element $(\theta_{gh})^{-1} \cdot h^*(\theta_g) \cdot \theta_h \in \mathbb{C}^*$ for any pair $g,h \in G$ and this assignment gives an element in $H^2(C,\mathbb{C}^*).$ When we take determinant of each $\theta_g$ we have $((\theta_{gh})^{-1} \cdot h^*(\theta_g) \cdot \theta_h)^2$ which gives the trivial element of $H^2(G,\mathbb{C}^*)$ since $\xi$ is a $G$-equivariant line bundle. Because $gcd(2,|H^2(G,\mathbb{C}^*)|)=1$ we see that $\theta_g$ gives a trivial element in $H^2(G,\mathbb{C}^*).$ Therefore $U$ is a $G$-equivariant vector bundle on $X.$
\end{proof}

Then we can construct orbifold Fano hosts of elliptic surfaces with $\kappa=1$ constructed above.

\begin{prop}
An elliptic surface $Y=(E \times C)/G$ constructed above where $|G| \leq 3$ and there is a $G$-equivariant odd degree line bundle on $C.$ Then $Y$ has an orbifold Fano host.
\end{prop}
\begin{proof}
From Lemma \ref{l.equiv.emb.ell.curve} we see that $E$ has a Fano host $F_1$ with $G$-action and and the Fourier-Mukai kernel $K_1$ is a $G$-linearized object with respect to the diagonal action. From the assumption there is a $G$-equivariant odd degree line bundle on $C.$ Again from the above two Lemmas \ref{l.invuniv}, \ref{l.equivuniv}, we see that $C$ has a Fano host $F_2$ with $G$-action and and the Fourier-Mukai kernel $K_2$ is also a $G$-linearized object with respect to the diagonal action. From the Theorem \ref{t.KSP}, we have a fully faithful functor $D^b(Y) \simeq D^b([(E \times C)/G]) \to D^b([(F_1 \times F_2)/G]).$ Therefore we obtain an orbifold Fano host of $Y.$
\end{proof}

We expect to obtain many more examples orbifold Fano hosts of surfaces with $\kappa=1$ via the above method.

\subsection{$\kappa = 2$ case}\label{s6.4}

Surfaces of general type are still mysterious objects. A very simple way to construct surfaces of general type is to consider complete intersection in projective spaces or product of two curves. From Remark \ref{Nar}, it is very easy to see that they are Fano visitors.

\begin{coro}
A surface which is a product of two curves is a Fano visitor.
\end{coro}

By Theorem \ref{t3.1}, we can provide many examples of surfaces of general type which are Fano visitors. However we do not know whether all surfaces of general type are Fano visitors or not, since many of them, e.g. surfaces of general type with $p_g=q=0,$ cannot be embedded in projective spaces as complete intersections. 

Recently, interesting new categories in the derived categories of surfaces of general type with $p_g=q=0$ were discovered (cf. \cite{BBS, BBKS, GS, Lee1, Lee2, LS}). Their Grothendieck groups are finite torsion and their Hochschild homology groups vanish. We call them \emph{quasi-phantom categories}. On the other hand, no Fano variety is known to have a quasi-phantom subcategory. Therefore the following question seems interesting.

\begin{ques}\label{q6.10}
Is there a Fano variety $X$ whose derived category contains a quasi-phantom category? 
\end{ques}

Obviously this question is closely related to the Fano visitor problem.

\begin{ques}\label{q6.11}
Let $Y$ be a surface of general type with $p_g=q=0$. Is there a Fano host of $Y$?
\end{ques}

For example, a Fano host of the determinantal Barlow surface will give us a Fano variety containing a phantom category.

Although we do not know the answer to Question \ref{q6.10}, we can construct a Fano orbifold whose derived category contains a quasi-phantom category.

\subsubsection{Classical Godeaux surfaces}

\begin{exam}\label{e6.13}
Let $Y \subset \PP^3$ be the variety defined by Fermat quintic $f=z_0^5+z_1^5+z_2^5+z_3^5=0$ and let $G=\ZZ_5=\langle \xi\rangle$ act on $Y$ by $\xi \cdot [z_0:z_1:z_2:z_3]=[z_0:\xi z_1:\xi^2 z_2:\xi^3 z_3]$ where $\xi=e^{\frac{2\pi \sqrt{-1}}{5}} $ is a primitive fifth root of unity. The $G$-action on $Y$ is free and $Y/G$ is the classical Godeaux surface. Let $X=w^{-1}(0)\subset \PP E^\vee$ be a Fano host of $Y=s^{-1}(0)\subset \PP^5$ obtained by the construction in \S\ref{s3.2} where $s$ is the section of $E=\cO_{\PP^5}(5)\oplus \cO_{\PP^5}(1)^{\oplus 2}$ defined by the Fermat quintic $f$ and two linear polynomials $z_4, z_5$ that cut out $\PP^3$ in $\PP^5$. Let $G$ act on $z_4$ and $z_5$ trivially. Then $G$ acts on $\PP^5$ and $E$ compatibly. Moreover the section $s=(f,z_4,z_5)$ is $G$-invariant. By Orlov's theorem (Remark \ref{r2.11}), we see that there is a fully faithful embedding $D^b(Y/G)\to D^b([X/G])$ of the derived category of the classical Godeaux surface into the derived category of the Fano orbifold $[X/G]$. Since the derived category of the classical Godeaux surface contains a quasi-phantom category (cf. \cite{BBS}), $D^b([X/G])$ also contains a quasi-phantom category. 
\end{exam}


\subsubsection{Product-quotient surfaces}

Let us briefly recall the definition of product-quotient surfaces.

\begin{defi}
An algebraic surface $S$ is called a product-quotient surface if there exist a fiinite group $G$ and two algebraic curves $C, D$ with $G$-action such that $S$ is isomorphic to the minimal resolution of $(C \times D)/G$ where $G$ acts on $C \times D$ diagonally.
\end{defi}

Product-quotient surfaces provide surprisingly many new examples of surfaces of general type and play an important role in the theory of algebraic surfaces (cf. \cite{BCGP}). Recently derived categories of some product-quotient surfaces were studied and it turns out that some of them have quasi-phantom categories in their derived categories(cf. \cite{GS,KKL,Lee1,Lee2,LS}). We can construct orbifold Fano hosts of some of product-quotient surfaces as follows.

\begin{prop}
Let $S$ be a product-quotient surface which is the minimal resolution of $(C \times D)/G.$ Suppose that $C, D$ have $G$-equivariant odd degree line bundles and $gcd(2,|H^2(G,\mathbb{C}^*)|)=1.$ Then $S$ has an orbifold Fano host.
\end{prop}
\begin{proof}
Let $C, D$ be algebraic curve with $G$-action such that $S$ is a minimal resolution of $(C \times D)/G.$ Then $D^b(S)$ is embedded into $D^b([(C \times D)/G])$ by McKay correspondence (cf. Theorem \ref{t.GL2}). From Lemma \ref{l.equivuniv} we see that $C$ (resp. $D$) has a Fano host $F_1$ (resp. $F_2$) with $G$-action and and the Fourier-Mukai kernel $K_1$ (resp. $K_2$) is a $G$-linearized object with respect to the diagonal action. From the Theorem \ref{t.KSP}, we have a fully faithful functor $D^b([(C \times D)/G]) \to D^b([(F_1 \times F_2)/G]).$ Finally $[(F_1 \times F_2)/G]$ is a smooth Deligne-Mumford stack whose coarse moduli space $(F_1 \times F_2)/G$ is a Fano variety. Therefore we get the desired result.
\end{proof}

\begin{exam}
Let $S$ be a product-quotient surface where the order of $G$ is odd. Then $S$ satisfies the conditions of the above theorem. See \cite{BCGP, GS, Lee1} for examples of these surfaces.
\end{exam}

\begin{coro}
There are Fano orbifolds whose derived categories contain phantom categories.
\end{coro}
\begin{proof}
Let $S_1$ be the classical Godeaux surface and $S_2$ be the project-quotient surface obtained by the quotient of product two genus 4 curves with free $\mathbb{Z}_3^2$-action.
Let $\mathcal{X}_1$ be an orbifold Fano host of $S_1$ and $\mathcal{X}_2$ be an orbifold Fano host of $S_2$ where we know the existence from the above discussion. Then $\mathcal{X}_1 \times \mathcal{X}_2$ is an orbifold Fano host of $S_1 \times S_2.$
 
It was proved that $D^b(S_1)$ contains a quasi-phantom category in \cite{BBS} and $D^b(S_2)$ contains a quasi-phantom category in \cite{Lee1}. Then $D^b(S_1 \times S_2)$ contains a phantom category by the result of \cite{GO}. Therefore we have an example of Fano orbifold whose derived category contains a phantom category. Indeed, we can find more examples of such Fano orbifolds from the results of \cite{GS, Lee1}.
\end{proof}

\section{Discussions}

\subsection{Grothendieck groups}

It is easy to see that when a triangulated category $\mathcal{T}$ has a semiorthogonal decomposition $\mathcal{T} = \langle \mathcal{A}, \mathcal{B} \rangle$ then $K_0(\mathcal{T})=K_0(\mathcal{A}) \oplus K_0(\mathcal{B}).$ Therefore if $D^b(X)$ has a semiorthogonal decomposition whose component is $D^b(Y)$ then $K_0(X)$ should contain $K_0(Y)$ as a direct summand. There are many smooth projective varieties whose Grothendieck group contain finite abelian groups as direct summands. However it seems that we do not know that whether there is a smooth projective Fano variety whose Grothendieck group contains a finite abelian group as a direct summand. We thank Alexander Kuznetsov for the following remark.

\begin{rema}
There is a smooth projective Fano variety $X$ such that $K_0(X)$ contains a torsion subgroup.
\end{rema}
\begin{proof}
The Grothendieck group $K_0(E)$ of an elliptic curve $E$ has a torsion subgroup. Because $E$ is a Fano visitor (cf, \cite{KKLL}), there is a smooth projective Fano variety $X$ such that $K_0(E)$ is a subgroup of $K_0(X).$ Therefore there is a Fano variety whose Grothendieck group contains a torsion subgroup.
\end{proof}

Note that in the above case, torsion subgroups are not direct summands of $K_0(E).$ Therefore it is natural to ask the following questions.

\begin{ques}
(1) Let $A$ be a finite Abelian group. Is there a smooth projective Fano variety whose Grothendieck group contains $A$ as a direct summand? \\
(2) Is there a smooth projective Fano variety whose Grothendieck group contains a finite Abelian group as a direct summand?
\end{ques}

Note that the negative answers to the above questions can be obstructions to the Bondal's original Question \ref{q1.1}.

\begin{rema}
If the answer to the 2nd question is negative, then we see that there exists a smooth projective variety which is not a Fano visitor. If the answer to the 1st question is negative for a finite Abelian group $A$ which can be contained in the Grothendieck group of a smooth projective variety as a direct summand, then we have the same conclusion.
\end{rema}

\subsection{Toric varieties}

It seems interesting to consider the Fano visitor problem for toric varieties. Because many interesting problems about toric varieties can be described and solved via combinatorics of fans or polytopes, we expect that there should be a combinatorial approach to this problem.

\begin{ques}
Let $Y$ be a smooth toric variety. Is there a Fano host of $Y$ constructed by an explicit combinatorial method? Can we compute its Fano dimension?
\end{ques}

One can ask the same question for spherical varieties.

\subsection{Phantom categories}

From the theorem of \cite{GO} we see that there are Fano orbifolds containing phantom categories. However we do not know any single example of smooth projective Fano variety whose derived category contains a (quasi-)phantom category. Recently several examples of surfaces whose derived categories containing (quasi-)phantom categories were constructed. Fano hosts of these surfaces will give us examples of smooth projective Fano varieties whose derived categories contain (quasi-)phantom categories.

\begin{ques}
(1) Is there a smooth projective Fano variety whose derived category contains a (quasi-)phantom category? \\
(2) Is there a smooth projective Fano variety (or a Fano orbifold) whose derived category contains derived category of a determinantal Barlow surface (cf. \cite{BBKS})? \\
(3) Is there a smooth projective Fano variety (or a Fano orbifold) whose derived category contains derived category of an elliptic surface construced by Cho and Lee (cf. \cite{CL})? 
\end{ques}

It will be very interesting if one can see these phantom categories in the Landau-Ginzburg mirror of (orbifold) Fano hosts.

\subsection{Noncommutative varieties}

There are many examples of noncommutative varieties in derived categories of Fano varieties. For example, Kuznetsov proved there are K3 categories not equivalent to derived cateogies of K3 surfaces inside derived categories of cubic 4-folds. These K3 categories provide a natural explanation why many holomorphic symplectic varieties arise from cubic 4-folds. Noncommutative varieties also appear in derived categories of cubic 3-folds and interesting applications of these noncommutative varieties were found (cf. \cite{LMS2}).

It will be an interesting question which noncommutative varieties can be embedded into derived categories of Fano orbifolds. It is also an interesting problem to find another geometric description of these noncommutative varieties (cf. \cite{LMS2}).

\subsection{Applications and perspectives}

Recently many interesting new applications of semiorthogonal decompositions were found and it seems that many new will appear in near future. As we mentioned in the introduction, derived categories of Fano varieties always have nontrivial semiorthogonal decompositions and moduli problems related to Fano varieties are very interesting.

Especially, it seems that semiorthogonal decompositions of derived categories can be useful to understand moduli space of ACM or Ulrich bunldes on Fano varieties since sometimes the original moduli problem reduces to another moduli problem about a (possibly noncommutative) variety of smaller dimension (cf. \cite{CKL, LMS1, LMS2}).

Moreover it seems that to study Fano varieties whose derived categories contain derived categories of K3 surfaces is very helpful to understand and construct holomorphic symplectic varieties. Note that there is always a nontrivial family of rational curves on a give Fano variety and the moduil spaces of rational curves on Fano varieties are very interesting objects. Semiorthogonal decompositions can be also very useful to understand rational curves on the Fano varieties.

Therefore we think that Fano visitor problem can give a systematic approach to study these moduli problems. It is also very interesting to find applications of Fano visitor problem to arithmetic geometry, birational geometry and (homological) mirror symmetry.

\bigskip

\bibliographystyle{amsplain}

\end{document}